\documentclass[a4paper, 11pt]{article}

\usepackage{amsmath,amsthm,amssymb,enumerate}
\usepackage[margin=2cm]{geometry}
\usepackage{tabularx}
\usepackage{xspace}
\usepackage{mathrsfs}

\usepackage[hidelinks,bookmarksnumbered,bookmarksopen=true,]{hyperref}
\usepackage[capitalise]{cleveref}

\theoremstyle{plain}
\newtheorem{theorem}{Theorem}

\newtheorem{conjecture}[theorem]{Conjecture}
\newtheorem{corollary}[theorem]{Corollary}
\newtheorem{definition}[theorem]{Definition}

\newtheorem{lemma}[theorem]{Lemma}

\newtheorem{proposition}[theorem]{Proposition}

\newtheorem{observation}[theorem]{Observation}

\theoremstyle{definition}
\newtheorem{claim}{Claim}
\newtheorem*{claim*}{Claim}

\newenvironment{pocd}[1]{\begin{proof}[Proof of {C}laim #1.]}{\end{proof}}

\DeclareMathOperator{\lrwd}{\mathsf{lrwd}}

\DeclareMathOperator{\rk}{\mathsf{rk}}
\DeclareMathOperator{\cutrk}{\mathsf{cutrk}}
\DeclareMathOperator{\rep}{\mathsf{rep}}
\DeclareMathOperator{\Twin}{\mathsf{Twin}}
\DeclareMathOperator{\Rep}{\mathsf{\mathcal{R}}}
\DeclareMathOperator{\paths}{\mathsf{paths}}
\DeclareMathOperator{\lettericity}{\mathsf{letr}}

\def\ie{i.e.\xspace}
\def\wrt{\emph{w.r.t.}\xspace}
\def\eg{\emph{e.g.}\xspace}

\def\mso{\texttt{MSO}\xspace}

\def\edg{\mathsf{edg}}

\def\cA{\mathcal{A}}

\def\cC{\mathcal{C}}

\def\Ind{\mathsf{Ind}}

\def\compl#1{\overline{#1}}

\newcommand{\CG}{\mathsf{CG}}
\newcommand{\alphabet}{\Sigma}

\DeclareMathOperator{\inter}{int}

\title{Lettericity of graphs: \\ an FPT algorithm and a bound on the size of obstructions}
\author{Bogdan Alecu \and Mamadou Moustapha Kant\'{e} \and Vadim Lozin  \and Viktor Zamaraev}

\date{}

\begin{document}
	
\maketitle

\begin{abstract}
Lettericity is a graph parameter responsible for many attractive structural properties. 
In particular, graphs of bounded lettericity have bounded linear clique-width and they are 
well-quasi-ordered by induced subgraphs. The latter property implies that any hereditary class 
of graphs of bounded lettericity can be described by finitely many forbidden induced subgraphs.
This, in turn, implies,  in a non-constructive way,  polynomial-time recognition of such classes.
However, no constructive algorithms and no specific bounds on the size of forbidden graphs 
are available up to date. In the present paper, we develop an algorithm that recognizes 
$n$-vertex graphs of lettericity at most $k$ in time $f(k)n^3$ and show that any minimal graph 
of lettericity more than $k$ has at most $2^{O(k^2\log k)}$ vertices. 
\end{abstract}

 \section{Introduction}
 
 Lettericity is a graph parameter introduced  by Petkov\v sek \cite{Petkovsek2002} in 2002.
 This notion remained unnoticed until 2011, and since then it has attracted much attention 
 in the literature due to several reasons. One of them is that graphs of
 bounded lettericity are well-quasi-ordered by the induced subgraph relation, which is not 
 the case for general graphs. An additional reason for the interest in lettericity is its relation to 
 other notions within graph theory and beyond. In particular, when restricted to permutation 
 graphs, bounded lettericity becomes equivalent to geometric griddability of permutations, as 
 was shown in \cite{AlecuFKLVZ2022}. 
 
 In the world of graph parameters, lettericity is situated between neighbourhood diversity
 and linear clique-width in the sense that bounded neighbourhood diversity implies bounded
 lettericity, which in turn implies bounded linear clique-width, but neither of these implications
 is reversible in general. Graphs of bounded neighbourhood diversity have a simple structure,
 which leads to an easy computation of this parameter. However, computing linear clique-width
 is NP-complete. More specifically, given a graph $G$ and a positive integer $k$, it is NP-complete
 to decide whether the linear clique-width of $G$ is at most $k$ \cite{CW}.
 On the other hand, for a fixed value of $k$, the computational complexity of recognizing 
 graphs of linear clique-width at most $k$ is an open problem for all $k\ge 4$ \cite{LCW3}. 

 For lettericity, the latter problem is known to be polynomial-time solvable, which follows from 
 the well-quasi-orderability of graphs of bounded lettericity under the induced subgraph relation.
 This property implies that, for each $k$, the class of graphs of lettericity at most $k$ can be
 described by a finite set of minimal forbidden induced subgraphs (obstructions). However, no
 constructive algorithms to solve the problem in polynomial time and no specific bounds on the size
 of obstructions are available to date. In the present paper, we fix this issue and present 
 the first polynomial-time algorithm to recognize graphs of lettericity at most $k$, for each 
 fixed value of $k$. Moreover, we show that this recognition problem is fixed-parameter tractable,
 when parameterized by $k$. We also provide an upper bound on the size of obstructions. 
 
 The organization of the paper is as follows. In Section~\ref{sec:pre}, we introduce basic
 terminology and notation used in the paper. In Section~\ref{subsec:mso-definability}, we show 
 that bounded lettericity is an MSO-definable property and derive from this conclusion 
 an implicit bound on the size of obstructions and the existence of an FPT algorithm that
 recognises $n$-vertex graphs of lettericity at most $k$ in time $f(k)n^3$. An alternative dynamic programming algorithm which does not use Courcelle's algorithm (and with a better function $f(k)$ than the one from Section~\ref{subsec:mso-definability}) can be found in the appendix.  
 In Section~\ref{sec:size-obs}, we obtain an explicit bound on the size of obstructions 
 by showing that any minimal graph of lettericity more than $k$ has at most $2^{O(k^2\log k)}$
 vertices. Section~\ref{sec:con} concludes the paper with a number of open questions.
	
	\section{Preliminaries}
    \label{sec:pre}

	
	\paragraph{Common notation.}
	We denote by $\mathbb{N}$ the set of non-negative integers, and for a positive integer $n$, we let $[n]$ be the set $\{1,\ldots, n\}$. If $M$ is a matrix
	with rows and columns indexed by a set $V$, and $X,Y\subseteq V$, we denote by $M[X,Y]$ the sub-matrix of $M$ whose rows and columns are indexed, respectively,
	by $X$ and $Y$. A \emph{partition} of a set $V$ is a collection of disjoint non-empty subsets of $V$ that cover $V$, and the elements of a partition are
	called \emph{blocks}.  An \emph{ordered partition} of $V$ is a sequence $(V_1, \ldots, V_k)$, for some positive integer $k$, such that $\{V_1,\ldots, V_k\}$ is
	a partition of $V$. For an ordered partition $\bar{V}=(V_1,\ldots,V_k)$ of $V$, let $\ell_{\bar{V}}:V\to [k]$ be the function such that $\ell_{\bar{V}}(x)=i$ if $x\in V_i$.
	
	For a directed or undirected graph $G$, we denote by $V(G)$ the vertex set of $G$.
	For an undirected graph $G$, the set of its edges is denoted by $E(G)$. If two vertices $x, y$ of $G$ are connected by an edge, we denote this edge as $xy$ or $yx$.
	The neighbourhood of a vertex $x$ in $G$ is denoted by $N_G(x)$, or $N(x)$ whenever $G$ is clear from the context. 
	A vertex $x \in V(G)$ \emph{distinguishes} vertices $y,z \in V(G)$ if it is adjacent to one of them
	and is not adjacent to the other; otherwise, $x$ does not distinguish $y$ and $z$.
	A connected component of $G$ is a maximal connected subgraph of $G$.
	For disjoint subsets $X$ and $Y$ of $V(G)$, we let $G[X]$ and
	$G[X,Y]$ be, respectively, the subgraphs $(X,(X\times X)\cap E(G))$ and $(X\cup Y,(X\times Y)\cap E(G))$. The complement of a graph $G$, denoted by $\overline{G}$, is the graph with vertex set $V(G)$, and
	$xy\in E(\overline{G})$ if and only if $xy\notin E(G)$. If $G$ is a bipartite graph, then we denote it by $(X,Y,E)$ where $\{X,Y\}$ is the bipartition of its
	vertex set into two independent sets. A bipartite graph $G=(X,Y,E)$ is a \emph{chain graph} if $X$ (equivalently $Y$) admits a linear ordering $x_1,x_2,...,x_n$
	such that $N_G(x_n)\subseteq N_G(x_{n-1}) \subseteq \cdots \subseteq N_G(x_2)\subseteq N_G(x_1)$; and $G$ is \emph{homogeneous} if it is either complete bipartite or edgeless.

	For a \emph{directed graph} $G$, the set of its directed edges, aslo called \emph{arcs}, which are ordered pairs of vertices, is denoted as $A(G)$. 
	An arc from $x$ to $y$ is denoted as $(x,y)$.
	For a vertex $x$, its \emph{in-neighbour set} is  the set $N_G^-(x):=\{y\mid (y,x)\in A(G)\}$ and its \emph{out-neighbour set} is $N_G^+(x)=\{y\mid (x,y)\in
	A(G)\}$. A \emph{circuit} in a directed graph $G$ is a sequence $(x_1,\ldots,x_p)$ of distinct vertices such that $(x_i,x_{i+1})\in A(G)$, for all $1\leq i \leq
	p-1$, and $(x_p,x_1)\in A(G)$. 
	A linear ordering $\leq$ of the vertices of a directed graph $G$ is called a \emph{topological ordering of $G$} if for every directed edge $(x,y)$ it holds that $x \leq y$.
	The \emph{underlying} graph of $G$ is the undirected graph obtained from $G$ by first removing edge orientations, and then removing multiple edges and loops, if any.
	A connected component in a directed graph $G$ is the directed subgraph of $G$ induced by the vertex set of a connected component of the underlying graph of $G$.
	
	For a directed or undirected graph $G$ and a set of vertices $S \subseteq V(G)$, we denote by $G[S]$ and $G\setminus S$ the subgraphs of $G$ induced by $S$ and $V(G)\setminus S$, respectively.
	
	A \emph{class of graphs}, also known as a \emph{graph property}, is a set of graphs closed under isomorphism.
	A graph class is \emph{hereditary} if it is closed under removing vertices.
	Any hereditary class can be defined by a finite or infinite set of \emph{obstructions},
	i.e.\ minimal graphs that do not belong to the class. More precisely, a graph $H$ is an obstruction for a hereditary class $\mathcal{C}$ if $H \not\in \mathcal{C}$, but $H\setminus \{v\} \in \mathcal{C}$ for every vertex $v$ of $H$.


	\paragraph{Letter graphs.}
	Let $\alphabet$ be a finite alphabet. A directed graph $D = (\alphabet, A)$ is called a \textit{decoder}
	over the alphabet $\alphabet$.
	An $n$-vertex graph $G=(V,E)$ is called a \textit{letter graph over decoder} $D=(\alphabet,A)$ if there exists
	a \textit{letter assignment} $\ell : V \rightarrow \alphabet$ and an \textit{ordering} $c: V \rightarrow [n]$
	of the vertices of $G$ such that two vertices $x,y \in V$ are adjacent in $G$
	if and only if:
	\begin{enumerate}
		\item either $(\ell(x), \ell(y)) \in A$ and $c(x) < c(y)$, 
		\item or $(\ell(y), \ell(x)) \in A$ and $c(y) < c(x)$.
	\end{enumerate}
	We say that the pair $(\ell, c)$ is a \textit{letter realisation} of $G$
	over decoder $D$, and that $G$ is a letter graph over decoder $D$. 
	If $|\alphabet| = k$, then $(\ell, c)$ is called a \emph{$k$-letter realisation} of $G$.
	A graph $G$ is a \emph{$k$-letter graph} if it admits a $k$-letter realisation.
	The \emph{lettericity} $\lettericity(G)$ of $G$ is the minimum $k$ such that $G$ is a $k$-letter graph. 
	
	For a letter realisation $(\ell,c)$ of an $n$-vertex graph $G$ over $D=(\Sigma,A)$, we denote by $w(\ell,c)$ the word $w_1w_2\cdots w_n$ over $\Sigma$ with
	$w_i=\ell(c^{-1}(i))$. 
	A \emph{word realisation of $G$ over $D$} is any word $w$ such that 
	$w=w(\ell,c)$ for some letter realisation $(\ell,c)$ of $G$ over $D$.

	We denote by ${\cal L}_D$ the class of all letter graphs over decoder $D$, and by
	${\cal L}_k$ we denote the class of all $k$-letter graphs, that is
	\begin{equation}\label{eq:union}
		{\cal L}_k = \bigcup_{D=(\alphabet,A) : |\alphabet| \leq k} {\cal L}_D.
	\end{equation}
	
	\noindent
	Two decoders $D_1$ and $D_2$ are \textit{equivalent} if ${\cal L}_{D_1} ={\cal L}_{D_2}$.
	In this case we write $D_1 \sim D_2$.
	In a directed graph, two vertices are called \emph{twins} if they have the same sets of in- and out-neighbours.
	The following observation is straightforward 
	
	\begin{observation}\label{obs:twin-decoder}
		Let $D=(\alphabet,A)$ be a decoder, and let $x,y \in \alphabet$ be twins in $D$. Then $D \sim D \setminus \{x\}$.
	\end{observation}
	
	Due to this observation, where convenient, we will consider twin-free decoders without explicitly referring to the observation.
	The next lemma can be derived from the definition of lettericity, and proofs can be found in \cite{Petkovsek2002}.

	\begin{lemma}[\cite{Petkovsek2002}]\label{lem:chain-graph} Let $G$ be an undirected graph.
		\begin{enumerate}
			\item For every letter realisation $(\ell,c)$ of $G$ over a decoder $D=(\alphabet,A)$,
			and any two distinct $a,b\in \alphabet$,
			\begin{itemize}
				\item $\ell^{-1}(a)$ is either a clique or an independent set in $G$;
				\item $G[\ell^{-1}(a),\ell^{-1}(b)]$ is a chain graph.
			\end{itemize} 
			\item For every $x \in V(G)$, $\lettericity(G\setminus \{x\}) \leq \lettericity(G) \leq 2\cdot \lettericity(G\setminus \{x\})+1$.
			\item $\lettericity(\overline{G})=\lettericity(G)$. 
		\end{enumerate}
	\end{lemma}

	The notion of \emph{rank-width} \cite{OumS2006} was introduced in order to approximate \emph{clique-width} of undirected graphs \cite{CourcelleO2000}, and
	\emph{linear rank-width} is the linearized version of rank-width. Given an $n$-vertex graph $G$ and $X\subseteq V(G)$, we let
	$\cutrk_G(X) = \rk(M_G[X,\compl{X}])$, where $M_G$ is the adjacency matrix of $G$ and $\rk$ is the matrix-rank function over the binary field.
	Given a linear order $\pi:V(G)\to [n]$ of an $n$-vertex graph $G$, we define the \emph{width of $\pi$} as
	$$ 
	\max\limits_{1\leq i \leq n-1}\cutrk_G(\pi^{-1}([i])),
	$$ 
	The \emph{linear rank-width} of a graph $G$, denoted by
	$\lrwd(G)$, is the minimum width over all linear orderings of $V(G)$. We refer for instance to \cite{AdlerK2013} and \cite{JeongKO17} for more information on
	computation of linear rank-width.
	The following proposition relates lettericity to \emph{linear rank-width}.
	
	\begin{proposition}\label{prop:lettericity-lrw} 
		For every graph $G$, $\lrwd(G) \leq \lettericity(G)$. 
	\end{proposition}
	
	\begin{proof} Let $(\ell,c)$ be a $k$-letter realisation of $G$. For every $1\leq i < n$, let $X_i:=c^{-1}([i])$ and $Y_i:=V(G)\setminus X_i$. By definition of
		a letter realisation, for every $x,y\in X_i$ with $\ell(x)=\ell(y)$, and $z\in Y_i$, we have $xz\in E(G)$ if and only if $yz\in E(G)$, i.e. $z$ does not distinguish $x$ and $y$.
		This implies that, for every $1\leq i<n$, there are at most $k$ different rows in the matrix $M_G[X_i,Y_i]$, and thus its rank is at most $k$.
	\end{proof}

	It is worth mentionning that linear rank-width and lettericity are not equivalent, \eg, paths have unbounded lettericity \cite{Petkovsek2002}, but their
        linear rank-width is $1$.

	\section{\mso-definability of $k$-lettericity}
    \label{subsec:mso-definability}

	In this section, we prove that, for a fixed $k$, the class of $k$-letter graphs is \mso-definable, i.e.\ there exists an \mso-formula that is satisfied by a graph $G$ if and only if $G$ is a $k$-letter graph.
	As a first consequence of this result, we obtain an FPT algorithm, parametrized by $k$, for recognizing $k$-lettericity, and more generally for recognizing any \mso-definable class of $k$-letter graphs. As a second consequence, we derive an implicit bound on the size of obstructions for any \mso-definable subclass of $k$-letter graphs. 

	We start by recalling some relevant notions; for more information on \mso formulas, we refer the reader to \cite{CourcelleE2012}. We assume that there is a countable
	set of vertex variables (we always use lower case letters to denote vertex variables), and another countable set of \emph{set (monadic) variables} (upper case letters
	are used for set variables).  The following are atomic formulas: $x=y$, $\edg(x,y)$ and $x\in X$. For \mso formulas $\varphi$ and $\psi$, the following are also
	\mso formulas: $\varphi \wedge \psi$, $\varphi \vee \psi$, $\neg \varphi$, $\exists x\varphi$ and $\exists X \varphi$.  We often write
	$\varphi(\bar{x},\bar{X})$, with $\bar{x}=(x_1,\ldots,x_p)$ and $\bar{X}=(X_1,\ldots, X_t)$, to mean that $x_1,\ldots, x_p,X_1,\ldots,X_t$ are free
	variables of $\varphi$, \ie, those that are not bound by the quantifier $\exists$. The meaning of \mso formulas is defined inductively as follows: 

    \noindent
    Let $G$ be
	a graph and $\varphi(\bar{x},\bar{X})$ an \mso formula. We say that $G$ \emph{satisfies} $\varphi(\bar{u},\bar{U})$, where $\bar{u}\in V(G)^p$ and
	$\bar{U}\in (2^{V(G)})^t$, whenever
	\begin{itemize}
		\item $u_i$ and $u_j$ are the same vertex in $G$, if $\varphi$ is $x_i=x_j$,
		\item $u_iu_j\in E(G)$, if $\varphi$ is $\edg(x_i,x_j)$,
		\item $u_i\in U_j$ if $\varphi$ is $x_i\in X_j$,
		\item $G$ satisfies $\varphi_1(\bar{u},\bar{U})$ and $\varphi_2(\bar{u},\bar{U})$ (resp.\ $\varphi_1(\bar{u},\bar{U})$ or $\varphi_2(\bar{u},\bar{U})$), if
		$\varphi$ is $\varphi_1\wedge \varphi_2$ (resp.\ $\varphi_1\vee \varphi_2$),
		\item $G$ does not satisfy $\psi(\bar{u},\bar{U})$, if $\varphi$ is $\neg \psi$,
		\item there is a vertex $z$ (resp.\ a set $Z$ of vertices) such that $G$ satisfies $\psi(z,\bar{u},\bar{U})$ (resp. $\psi(\bar{u},Z,\bar{U})$), if $\varphi$ is
		$\exists y\psi(y,\bar{x},\bar{X})$ (resp.\ $\exists Y\psi(\bar{x},Y,\bar{X})$).
	\end{itemize}

	As usual, we use $\forall x \varphi(x)$ (resp. $\forall X \varphi(X)$) for $\neg (\exists x \neg \varphi(x))$ (resp. $\neg (\exists X \neg \varphi(X))$),
	$\varphi\to \psi$ for $\neg \varphi \vee \psi$ and $\varphi \leftrightarrow \psi$ for $\varphi \to \psi \wedge \psi \to \varphi$. We also use the shortcut
	$\exists!x \varphi(x)$ for $\exists x(\varphi(x) \wedge \forall y (\varphi(y) \to y=x))$ and $X\subseteq Y$ for $\forall x(x\in X\to x\in Y)$. 
	A graph property $\mathcal{P}$ is \emph{\mso-definable} if there exists an \mso-formula $\varphi$
	such that a graph $G$ satisfies $\varphi$ if and only if $G$ belongs to $\mathcal{P}$.
	
	We first provide a
	combinatorial characterisation of $k$-lettericity  (Theorem \ref{thm:characterisation}) and then use it for constructing an \mso formula that characterises $k$-lettericity.
	
	\begin{definition}[Compatibility graph]\label{defn:compatible} 
		For a graph $G=(V,E)$, a decoder $D = (\alphabet, A)$,
		and a letter assignment $\ell : V \rightarrow \alphabet$,
		the \emph{compatibility graph}, denoted by $\CG(G,D,\ell)$, 
		is the directed graph with vertex set $V$ and the arc set
		\begin{align*}
			\bigcup\limits_{(a,b)\in A, (b,a)\notin A} \big\{ (x,y)\mid \ell(x)=a,\ell(y)=b,xy\in E \big\} \cup \big\{(y,x)\mid \ell(x)=a,\ell(y)=b,xy\notin E \big\}.
		\end{align*}
	\end{definition}
	
	\noindent
	The compatibility graph associated with  a letter assignment over $D$ 
	represents precedence order that any letter realisation $(\ell,c)$ over $D$ should satisfy.
	Namely, for any such letter realisation and distinct vertices $x,y \in V(G)$ we have 
	that if $(x,y)$ is an arc in $\CG(G,D,\ell)$, then $c(x) < c(y)$.

	\begin{theorem}\label{thm:characterisation}  Let $D=(\alphabet,A)$ be a $k$-letter decoder. 
		A graph $G=(V,E)$ is a $k$-letter graph over $D$ if and only if there exists a letter assignment $\ell : V \rightarrow \alphabet$ such that
		\begin{enumerate}
			\item[(C1)] for any $a\in \Sigma$ with $|\ell^{-1}(a)| \geq 2$, the set $\ell^{-1}(a)$ is a clique if and only if $(a,a)\in A$, 
			\item[(C2)] for any $a\in \Sigma$ with $|\ell^{-1}(a)| \geq 2$, the set $\ell^{-1}(a)$ is an independent set if and only if $(a,a)\notin A$,
			\item[(C3)] for any $a,b \in \Sigma$, 
			if $(a,b),(b,a)\in A$ (resp. $(a,b),(b,a)\notin A$), then the graph $G[\ell^{-1}(a), \ell^{-1}(b)]$ is a complete bipartite graph (resp.\ edgeless bipartite graph).
			\item[(C4)] the compatibility graph $\CG(G,D,\ell)$ is acyclic.
		\end{enumerate}
		
		Moreover, for any topological ordering $\pi$ of the compatibility graph, the pair $(\ell, \pi)$ is a $k$-letter realisation of $G$ over $D$.
	\end{theorem}
	
	\begin{proof} 
		Assume that $\ell$ is a letter assignment satisfying (C1)-(C4).
		Let $\pi$ be a topological ordering of $\CG(G,D,\ell)$ and let $G'$ be the graph with vertex set $V$ and with the $k$-letter realisation $(\ell,\pi)$ over $D$.  
		We claim that $G=G'$. Let $x$ and $y$ be two vertices in $V$. We first
		notice that if $\ell(x)=\ell(y)$, then $xy\in E(G')$ if and only if $(\ell(x),\ell(x))\in D$, and since $\ell$ satisfies (C1)-(C2), this is also the case in $G$. 
		Similarly, due to (C3), if $\ell(x) \neq \ell(y)$ and either both or none of $(\ell(x),\ell(y))$ and $(\ell(y),\ell(x))$ belong to $D$, then $xy\in E(G)$ if and only if $xy\in E(G')$. 
		We therefore can assume that $\ell(x)\ne \ell(y)$ and exactly one of $(\ell(x),\ell(y))$ and $(\ell(y),\ell(x))$ belongs to $D$.
		Suppose, without loss of generality, that $(\ell(x),\ell(y))\in D$. 
		If $xy \in E(G)$, then $(x,y)\in A(\CG(G,D,\ell))$ and hence $\pi(x)<\pi(y)$, which implies $xy\in E(G')$. 
		Similarly, if $xy \notin E(G)$, then $(y,x)\in A(\CG(G,D,\ell)$ and hence $\pi(y)<\pi(x)$, implying $xy \notin E(G')$.
		
		Assume now that $G$ is a $k$-letter graph and $(\ell,c)$ is a $k$-letter realisation of $G$ over $D$. By definition, $\ell$
		satisfies (C1)-(C3) and it remains
		only to show that $\CG(G,D,\ell)$ is acyclic. Suppose $\CG(G,D,\ell)$ has a circuit $(x_1,x_2,\ldots,x_p)$. From the definition of $\CG(G,D,\ell)$, we know that, for each $1\leq j\leq p$, with indices modulo $p$,
		\begin{itemize}
			\item either $x_{j}x_{j+1}\in E(G)$, $(\ell(x_{j}),\ell(x_{j+1}))\in D$ and $(\ell(x_{j+1}),\ell(x_{j}))\notin D$,
			\item or $x_{j}x_{j+1}\notin E(G)$, $(\ell(x_{j+1}),\ell(x_{j}))\in D$ and $(\ell(x_{j}),\ell(x_{j+1}))\notin D$.
		\end{itemize}
		Hence, $c(x_{j}) < c(x_{j+1})$ for all $1\leq j\leq p$. Therefore, $c(x_1)<c(x_2) < c(x_3)< \cdots < c(x_p) < c(x_1)$, a
		contradiction. 
	\end{proof}

	For a graph $G$, we will
	call any ordered partition $\bar{X}=(X_1,\ldots, X_k)$ of $V(G)$
	such that $\ell_{\bar{X}}$ satisfies (C1)-(C4) for some $D=([k],A)$, a \emph{letter partition of $G$ over $D$}. We now show how
	to define an \mso formula $\varphi_k$ such that a graph satisfies $\varphi_k$ if and only it has lettericity at most $k$.

	\begin{proposition}\label{prop:d-mso} 
		For every decoder $D=([k],A)$, there exists an \mso formula $\varphi_D(X,X_1,\ldots,X_k)$, whose length depends linearly on the size of $D$, such that
		for every graph $G$ and $Z\subseteq V(G)$, $G[Z]$ admits a $k$-letter realisation $(\ell,c)$ over $D$ if and only if $G$ satisfies
		$\varphi_D(Z,\ell^{-1}(1),\ldots, \ell^{-1}(k))$.
	\end{proposition}
	
	\begin{proof} 
		We first notice that any map $\ell:V(G)\to [k]$ corresponds to a partition of $V(G)$ into at most $k$ parts.  Moreover, by Theorem \ref{thm:characterisation},
		a graph $G$ has a $k$-letter realisation $(\ell,c)$ over $D$ if and only if $(\ell^{-1}(1), \ldots,\ell^{-1}(k))$ is a letter partition of $G$ over $D$.  We
		can express $\varphi_D(X,X_1,\ldots,X_k)$ as follows, which decides the four conditions (C1)-(C4):
		\begin{align*}
			& Partition(X,X_1,\ldots,X_k) \wedge \\
			& \left(\bigwedge\limits_{(a,b),(b,a)\in D} \forall x,y (x \in X_a \wedge y\in X_b \rightarrow \edg(x,y))\right) \wedge \\
			& \left(\bigwedge\limits_{(a,b),(b,a)\notin D} \forall x,y (x \in X_a\wedge y\in X_b \rightarrow \neg\edg(x,y))\right) \wedge \\
			& Acyclic(X,X_1,\ldots,X_k),
		\end{align*}
		
		\noindent
		where the formula $Partition(X,X_1,\ldots,X_k)$ checks that $\{X_1, \ldots, X_k\}$ is a partition of $X$, and the formula $Acyclic(X,X_1,\ldots,X_k)$ is defined as follows:
		\begin{align*}
			\neg\Big(\exists Y(Y\subseteq X\ \wedge\ \forall x \in Y \exists!y \exists! z (y\in Y \wedge z\in Y \wedge \psi(x,y,X,X_1,\ldots,X_k) \wedge \psi(z,x,X,X_1,\ldots,X_k)))\Big),\\
			\intertext{where $\psi(x,y,X,X_1,\ldots,X_k)$ is equal to}
			x\in X \wedge y\in X \wedge\ \bigvee\limits_{(a,b)\in D,(b,a)\notin D} (x\in X_a \wedge y\in X_b \wedge \edg(x,y)) \vee (x\in X_b\wedge y\in X_a \wedge \neg \edg(x,y)).
		\end{align*}
		
		\noindent 
		Clearly $\varphi_D$ has the desired length.  The first three parts of $\varphi_D$ check that $\bar{X}=(X_1,\ldots, X_k)$ is an ordered partition of
		$X$ such that whenever $(a,b),(b,a)\in D$ (resp. $(a,b),(b,a)\notin D$), then $G[X_a,X_b]$ is a complete (resp. edgeless) bipartite graph.  Notice that
		the formula also checks that each part is either a clique (if $(a,a)\in D$) or is an independent set (if $(a,a)\notin D$) because $a$ is not necessarily
		different from $b$ in the big conjunctions.  Observe that all these checks decide whether the letter assignment $\ell_{\bar{X}}$ satisfies (C1)-(C3). Now,
		it is easy to verify that $\psi(x,y,X,X_1,\ldots,X_k)$ corresponds to the arcs of $\CG(G[X],D,\ell_{\bar{X}})$ by the definition of compatibility graph. Indeed,
		$(x,y)\in A(\CG(G[X],D,\ell_{\bar{X}}))$ if ($(\ell_{\bar{X}}(x),\ell_{\bar{X}}(y))\in D$, $(\ell_{\bar{X}}(y),\ell_{\bar{X}}(x))\notin D$ and
		$xy\in E(G)$) or ($(\ell_{\bar{X}}(y),\ell_{\bar{X}}(x))\in D$, $(\ell_{\bar{X}}(x),\ell_{\bar{X}}(y))\notin D$ and $xy\notin E(G)$), and this
		is exactly what $\psi(x,y,X,X_1,\ldots,X_k)$ checks.  Finally, the sub-formula $Acyclic(X,X_1,\ldots,X_k)$ checks the non-existence of a subset $Y$ of $X$
		such that in the subgraph of $\CG(G[X],D,\ell_{\bar{X}})$ induced by $Y$ every vertex has in-degree $1$ and out-degree $1$, which characterises the existence
		of a circuit. We can therefore conclude that $G[Z]$ has a $k$-letter realisation $(\ell,c)$ over $D$ if and only if $G$ satisfies
		$\varphi_D(Z,\ell^{-1}(1),\ldots, \ell^{-1}(k))$.
	\end{proof}
	
	We are now ready to define a formula, which is satisfied by graphs of lettericity at most $k$ and only by them.
	
	\begin{theorem}\label{thm:mso} For every positive integer $k$, there exists an \mso formula $\varphi_k(X,X_1,\ldots,X_k)$ of length $2^{O(k^2)}$,
		such that, for every graph $G$ and $Z\subseteq V(G)$, $G[Z]$ is a $k$-letter graph if and only if $G$ satisfies $\varphi_k(Z,Z_1,\ldots,Z_k)$ for some ordered
		partition $(Z_1,\ldots,Z_k)$ of $Z$.
	\end{theorem}
	\begin{proof} 
		By Proposition~\ref{prop:d-mso}, for every decoder $D=([k],A)$, one can construct a formula 
		$$
		\varphi_D(X,X_1,\ldots,X_k),
		$$
		whose length depends linearly on the size of $D$,
		such that $G[Z]$ admits a $k$-letter realisation $(\ell,c)$ over $D$ if and only if $G$ satisfies
		$\varphi_D(Z,\ell^{-1}(1),\ldots,\ell^{-1}(k))$. 
		Thus, by defining $\varphi_k (X,X_1,\ldots, X_k)$ as 
		\begin{align*}
			\bigvee_{\substack{D\ \textrm{: directed graph} \\ \textrm{on $k$ vertices}}} \varphi_D(X,X_1,\ldots, X_k),
		\end{align*}
		we conclude that $G[Z]$ is a $k$-letter graph if and only if $G$ satisfies $\varphi_k (Z,Z_1,\ldots, Z_k)$
		for some ordered partition $(Z_1,\ldots, Z_k)$ of $Z$.
		
		The bound on the length of $\varphi_k$ follows from the fact that there are at most $2^{k^2}$ directed graphs on $k$ vertices.
	\end{proof}
	
	Using \cref{prop:lettericity-lrw} and \cref{thm:mso}, we can provide an FPT algorithm for recognising $k$-letter graphs. An algorithm with a better
        running time \wrt $k$ is proposed in Appendix \ref{sec:fpt-recognition}.
	
	\begin{theorem}\label{thm:recognition}  Let $k$ be a positive integer.  For some computable function $f:\mathbb{N}\to \mathbb{N}$, there is an algorithm that takes as input an $n$-vertex graph $G$ and checks in time
		$f(k)\cdot n^3$ whether $G$ has lettericity at most $k$, and if so outputs a $k$-letter realisation of $G$. The same holds for any \mso-definable hereditary property of $k$-letter graphs. 
	\end{theorem}
	\begin{proof} 
		Let $G$ be an $n$-vertex graph. The algorithm works in two phases.
		In the first phase, it checks whether $G$ has linear rank-width at most $k$ using the FPT algorithm from \cite{JeongKO17}. In time $2^{O(k^2)}\cdot n^3$, this algorithm outputs NO if the linear rank-width of $G$ is larger than $k$, and otherwise it outputs 
		a linear ordering $x_1,x_2,\ldots,x_n$ of the vertices of $G$ witnessing linear rank-width at most $k$.
		In the former case, by Proposition \ref{prop:lettericity-lrw}, our algorithm correctly answers that the lettericity of $G$ is strictly greater than $k$.
		In the latter case, the algorithm proceeds to the second phase. In this phase, 
		for every $k$-letter decoder $D$, we invoke
		Courcelle's algorithm which, given the linear
		ordering $x_1,\ldots, x_n$ of the vertices of $G$, checks in time $g(k,|\varphi'_D|)\cdot n$, for some function $g$, whether there exists $\bar{Z} = (Z_1, Z_2, \ldots, Z_k)$ such that $G$ satisfies
		$\varphi'_D(Z_1, Z_2, \ldots, Z_k)$, and if so computes one such $\bar{Z}$ (see for instance \cite[Theorem 6.55]{CourcelleE2012} or \cite{CourcelleK2009}),
		where
		$$
		\varphi_D'(X_1,\ldots, X_k) := \exists X\left(\forall x(x\in X)\ \wedge\ \varphi_D(X,X_1,\ldots, X_k)\right).
		$$
		If $G$ does not satisfy any $\varphi_D'$, then it follows from Proposition \ref{prop:d-mso} that $G$ is not a $k$-letter graph. Otherwise,
		our algorithm picks a decoder $D$ and an ordered partition $\bar{Z} = (Z_1, \ldots, Z_k)$ such that $G$ satisfies $\varphi_D'(Z_1, Z_2, \ldots, Z_k)$. 
		Then, by Proposition \ref{prop:d-mso} and Theorem \ref{thm:characterisation}, the letter assignment $\ell_{\bar{Z}}$ together with any
		topological ordering $\pi$ of the compatibility graph $\CG(G,D,\ell_{\bar{Z}})$ is a $k$-letter realisation of $G$. Both $\ell_{\bar{Z}}$ and $\pi$ can be computed in polynomial time from $G, D$, and $\bar{Z}$, which concludes the proof of the first statement.
		
		For the second statement, if $\cC$ is a hereditary family of $k$-letter graphs \mso-definable by $\theta_{\cC}$, 
		the conclusion follows using the same algorithm as for the first statement and by replacing
		$\varphi'_D$ with $\varphi'_D \wedge \theta_{\cC}$.
	\end{proof}
	
	Using the formula $\varphi_k$ from \cref{thm:mso}, we can define 
	an \mso formula that characterises obstructions for
	$k$-lettericity.  We remind that a graph $H$ is an \emph{obstruction} for a hereditary property $\mathcal{P}$ if $H$ does not satisfy $\mathcal{P}$, but $H\setminus \{x\}$
	does for every vertex $x$ of $H$.
	
	\begin{proposition}\label{prop:obs} There exists an \mso sentence $Obs_k$ of length  $2^{O(k^2)}$, such that a graph $G$ is an obstruction for the class of
		$k$-letter graphs if and only if $G$ satisfies $Obs_k$.
	\end{proposition}
	\begin{proof} 
		By definition, a graph $G$ is an obstruction for $k$-lettericity if and only if $G$ is not a $k$-letter graph, but $G\setminus \{x\}$ is a $k$-letter graph for every $x\in V(G)$. It is
		straightforward now to check that $G$ is an obstruction for $k$-letter graphs if and only if $G$ satisfies the following \mso sentence $Obs_k$
		\begin{align*}
			& \neg \left(\exists X,X_1,\ldots, X_k (\forall x(x\in X) \wedge \varphi_k(X,X_1,\ldots,X_k))\right) \wedge\\
			& \forall x (\exists X,X_1, \ldots,X_k(x\notin X\wedge \forall y(y\ne x \to y\in X) \wedge \varphi_k(X,X_1,\ldots,X_k))).
		\end{align*}
	\end{proof}
	
	We can now derive an implicit bound on the size of obstructions for $k$-lettericity or for any other \mso-definable hereditary property of $k$-letter graphs.
	
	\begin{theorem}\label{thm:obs1} Let $k$ be a positive integer. There is a computable function $f:\mathbb{N}\to \mathbb{N}$ such that any obstruction for \emph{the class of $k$-letter graphs} has at most $f(k)$ vertices. The same holds for any \emph{\mso-definable hereditary property of $k$-letter graphs}. 
	\end{theorem}
	
	\begin{proof} Any word $w$ over an alphabet $\alphabet$ of length $n$ can be represented by the structure $([n], <, (P_a)_{a\in \alphabet})$ where $<$ is the
		natural order on integers and, for each $a\in \alphabet$, $P_a$ is the predicate equal to $\{i\in [n]\mid w_i=a\}$ (see for instance \cite{Buchi1960}). One can therefore write \mso formulas defining
		properties on words, using  the binary relation $<$ and predicates $(P_a)_{a\in \alphabet}$.
		
		One can easily check that if $w$ is a word realisation of a graph $G=([n],E)$ over a decoder $D=(\alphabet,A)$, then the following \mso formula on words decides whether the vertices $x$ and $y$ are adjacent in $G$
		$$ 
		\varphi_E^D(x,y) := \left(x < y\ \wedge\ \bigvee_{(a,b)\in A} P_a(x)\wedge P_b(y)\right)\  \vee\ \left(y < x\ \wedge\ \bigvee_{(a,b)\in A} P_a(y)\wedge P_b(x)\right).
		$$

		Thus, any \mso formula on $k$-letter graphs over $D$ can be translated into an \mso formula on word realisations over $D$ by replacing any occurrence of
		$\edg(x,y)$ by $\varphi_E^D(x,y)$ (see for instance \cite[Theorem 7.10]{CourcelleE2012}).
		
		Using now the fact that the set of obstructions for $k$-letter graphs is \mso-definable 
		(by \cref{prop:obs}) and each obstruction has lettericity at
		most $2k+1$ (which follows from \cref{lem:chain-graph}), we conclude the existence of an \mso formula $\psi_k$ on words over an alphabet with $2k+1$ letters
		such that a word satisfies $\psi_k$ if and only if it is a word realisation of an obstruction for $k$-letter graphs. By B\"{u}chi's Theorem \cite{Buchi1960},
		there is a finite state word automaton $\cA$ such that a word realisation satisfies $\psi_k$ if and only if it is accepted by $\cA$. By the Pumping Lemma on
		finite state word automaton (see e.g.\ \cite{Sip96}), any word accepted by $\cA$ and of size strictly greater than the number of states in $\cA$ has a proper subword accepted by
		$\cA$. Since any subword of a word realisation of $G$ is a word realisation of an induced subgraph of $G$, we conclude that the size of any obstruction is bounded from above by the number of states in $\cA$. 

		The second assertion is proved similarly. Let $\cC$ be a hereditary class of $k$-letter graphs \mso-definable by $\theta_{\cC}(X)$, where $G$ satisfies $\theta_{\cC}(Z)$ if
		and only if $G[Z] \in \cC$. Then, the set of obstruction for $\cC$ is \mso-definable by 
		the following formula $\psi_{\cC}$, which defines minimal graphs that do not satisfy $\theta_{\cC}(Z)$:
		\begin{align*}
			\psi_{\cC} := \neg(\exists X(\forall x(x\in X)\wedge \theta_{\cC}(X))\ \wedge
			\ \forall x (\exists X(x\notin X\wedge \forall y(y\ne x\to y\in X) \wedge \theta_{\cC}(X))).
		\end{align*}
		
		Since, as before, any obstruction for $\cC$ has lettericity at most $2k+1$, formula $\psi_{\cC}$
		can be translated into an \mso formula on words over an alphabet with $2k+1$ letters such
		that a word satisfies the formula if and only if it is a word realisation of an obstruction for $\cC$.
		As in the previous case, this implies a bound on the number of vertices in such obstructions.
	\end{proof}

	We note that while Theorem \ref{thm:obs1} guarantees a bound on the size of obstructions for $k$-lettericity that depends only on $k$, this dependence
	can involve iterated exponentials. Indeed, if $a$ is the maximum number of states of complete and deterministic automaton associated with atomic formulas,
	then it is proved in \cite[Corollary 6.30]{CourcelleE2012} that the number of states of the minimal automaton associated with any formula $\varphi$ can be bounded
	by $exp(h,m(a^m+h))$ where $h$ and $m$ are two values associated to $\varphi$, both upper-bounded by $|\varphi|$, and 
	$exp:\mathbb{N}\times \mathbb{N}\to \mathbb{N}$ is the function defined by $exp(0,n)=n$ and $exp(h+1,n)=2^{exp(h,n)}$.  In the next section, we will provide an explicit
	single exponential bound with a direct proof.
	
	We conclude this section with a remark on the relationship between bounded lettericity of graphs and geometric griddability of permutations.
	The notion of \emph{geometric griddability} was introduced in \cite{AlbertABRV2013} in order to study well-quasi-ordering and rationality of permutation classes.
	It is known that geometrically griddable classes of permutations are in bijection with \mso-definable trace languages \cite{AlbertABRV2013}, and it is proved in \cite{AlecuFKLVZ2022} that a class of permutations is geometrically griddable if and only if its associated permutation graph class has bounded lettericity. 
	Since geometrically griddable classes of permutations and graph classes of bounded lettericity are both well-quasi-ordered by the pattern containment relation and the induced subgraph relation respectively, one may be interested in other properties shared by the two concepts.
	In recent work \cite{braunfeld2023}, Braunfeld showed that geometric griddablity is \mso-definable
	and applied this result to answer the questions from \cite{AlbertABRV2013} about the computability of obstructions and 
	generating functions for geometric griddable classes of permutations. As a consequence, Braunfeld also established the 
	existence of an upper bound on the maximum size of an obstruction for a geometrically griddable class that depends only on the size of a geometric specification of the class.
	These results parallel ours about classes of graphs of bounded lettericity, and the proof strategies are very similar. We refer the reader to \cite{braunfeld2023} for further details.

	
	\section{A single exponential upper bound on the size of obstructions}
    \label{sec:size-obs}
	
	
	We have seen in Theorem \ref{thm:obs1} that any obstruction for $k$-lettericity has at most $f(k)$ vertices, for some function $f:\mathbb{N}\to \mathbb{N}$. The aim of this section is to show that $f(k)$ can be bounded from above by $2^{O(k^2\log k)}$. In \cref{sec:auxiliary-tools}, we prove a number of auxiliary results that we use in \cref{sec:upper-bound} to derive the bound.
	
	We say that a graph $G$ is a {\em critical letter graph} if, for any proper induced subgraph $H$ of $G$, $\lettericity(H) < \lettericity(G)$. 
	If, in addition, $\lettericity(G) = k$, we call $G$ a {\em critical $k$-letter graph}.
	A word $w'$ is a \emph{subword} of a word $w$, if $w'$ is obtained from $w$ by removing some letters.
	A \emph{factor} of $w$ is a \emph{contiguous} subword of $w$, \ie, a subword that can be obtained from $w$ by removing some
	(possibly empty) prefix and suffix.

	\subsection{Auxiliary tools}
	\label{sec:auxiliary-tools}
	
	In an undirected graph $G$, two vertices $x$ and $y$ are
	{\em twins} if $x$ and $y$ have the same neighbourhood in $V(G)\setminus \{x,y\}$.
	The twin relation is an equivalence relation. A \emph{twin class} is an equivalence class of this relation. Notice that if $X$ is a twin class, then $X$ is either a clique or an independent set.

 The overall intuition behind our proof of the bound is simple: we will show that a critical letter graph (in particular, an obstruction) cannot have long factors that use few letters. The next lemma is a base case for induction (which we will do in detail in Section 4.2), but it also serves as a proof of concept for our method.
	
	\begin{lemma}\label{lem:large-twin-classes} Let $G$ be a graph with a twin class $X$ of size at least $4$. Then, for every $x\in X$,
		$\lettericity(G)=\lettericity(G\setminus \{x\})$.
	\end{lemma}
	\begin{proof} Because $\lettericity(G\setminus \{x\})\leq \lettericity(G)$, it suffices to prove the inverse inequality. Also,
		since $\lettericity(\compl{H}) = \lettericity(H)$ for all graphs $H$ (see Lemma \ref{lem:chain-graph}), we can assume without loss of generality that $X$ is a clique. 
		Let $x\in X$ and assume that $\lettericity(G\setminus \{x\}) = k$, and let $(\ell,c)$ be a $k$-letter realisation of
		$G\setminus \{x\}$ over some decoder $D=([k],A)$.  Consider three vertices $x_1,x_2,x_3\in X\setminus \{x\}$ with $c(x_1) < c(x_2) < c(x_3)$, which exist because $|X|\geq 4$. 
		We analyse two cases:
		\begin{itemize}
			\item There exists an $i\in [3]$ such that $(\ell(x_i),\ell(x_i))\in A$. In this case, $(\ell_1,c_1)$ is a $k$-letter realisation of $G$, where
			\begin{align*}
				\ell_1(z) &=\begin{cases} \ell(z) & \textrm{if $z\ne x$},\\
					\ell(x_i) & \textrm{otherwise}. \end{cases}\\
				c_1(z) & = \begin{cases} c(z) & \textrm{if $z\ne x$ and $c(z)\leq c(x_i)$},\\
					c(z)+1 & \textrm{if $z\ne x$ and $c(z)> c(x_i)$},\\
					c(x_i)+1 & \textrm{otherwise}.
				\end{cases}
			\end{align*}
			
			\item For each $i\in [3]$, $(\ell(x_i),\ell(x_i))\notin A$. In this case, as $X$ is a clique, both $(\ell(x_1),\ell(x_2))$ and $(\ell(x_2),\ell(x_3))$ must be in
			$A$. Now, if there is a vertex $z\in V(G)\setminus \{x\}$ with $\ell(z)=\ell(x_2)$ and $c(z)<c(x_3)$, then $z$ would distinguish $x_2$ and $x_3$, which is impossible
			as they are twins. Similarly, there is no $z\in V(G)\setminus \{x\}$ with $\ell(z)=\ell(x_2)$ and $c(z)>c(x_1)$. Therefore, $x_2$ is the unique
			vertex $z$ in $V(G)\setminus \{x\}$ with $\ell(z)=\ell(x_2)$. Let $D'=([k],A\cup \{(\ell(x_2),\ell(x_2))\})$. 
			We note that $(\ell, c)$ is still a letter realisation of $G \setminus \{x\}$ over the decoder $D'$, and we can proceed as in the first case. 
		\end{itemize}
	\end{proof}
	
	When considering factors which use more than a single letter, things quickly get much more complicated. We would like to build a mechanism that generalizes \cref{lem:large-twin-classes}: if we have a factor using some number $t$ of letters, and that factor is very long, we would like to deduce that some of the letters appearing in the factor are superfluous, in the sense that their presence does not influence the lettericity of the graph, and hence the graph cannot be a critical letter graph. This requires a number of technicalities, which we will now introduce. 
	
    Let $D=(\Sigma,A)$ be a decoder. The \emph{asymmetry graph of $D$} is the directed graph
	with vertex set $\Sigma$, where, for every ordered pair $a,b\in \Sigma$, $(a,b)$ is an arc in the graph if and only if $(a,b) \in A$ and $(b,a) \not\in A$.
	If there is an arc between $a$ and $b$ in the asymmetry graph, we will say that $a$ and $b$ are \emph{dependent in $D$}, otherwise we will say that $a$ and $b$ are \emph{independent in $D$}. 
	To justify this terminology, we observe that the adjacency of two vertices labelled by independent letters is independent of their relative positions 
	in a word representing the graph. In other words, if two distinct letters $a$ and $b$ are independent, then, regardless of their positions in the word, they describe a bipartite graph which is either edgeless (if neither $(a,b)$ nor $(b,a)$ belongs to the decoder) or complete bipartite (if both $(a,b)$ and $(b,a)$ belong to the decoder).
	
	Let $w$ be a word over an alphabet $\Sigma$.
	For two distinct letters $a,b \in \Sigma$, we denote by $\inter_w(a, b)$ the largest $t$ such that $w$ contains a subword which is the $t$-fold concatenation of $ab$. We note that $\inter_w(a, b)$ and $\inter_w(b, a)$ differ by at most 1. We say that $a$ and $b$ \emph{interlace in $w$} if $w$ contains $abab$ or $baba$ as a subword, i.e $\max \{\inter_w(a, b), \inter_w(b, a) \} \geq 2$.
	The following observation follows directly from the definitions.

	\begin{observation}\label{obs:dependency}
		Let $(\ell,c)$ be a letter realisation of a graph $G$ over a decoder $D=(\Sigma,A)$, and
		let $a, b \in \Sigma$ be two distinct letters that interlace in $w(\ell,c)$. 
		Let $A,B \subseteq V(G)$ be the sets of vertices represented by the letters $a$ and $b$,
		respectively.
		Then $a$ and $b$ are independent in $D$ if and only if the bipartite graph $G[A,B]$ is homogeneous.
	\end{observation}

	\begin{lemma}\label{lem:interlacing} 
		Let $w$ be a word over an alphabet $\Sigma$, and let $x,y,z$ be pairwise distinct letters in $\Sigma$. If there is a copy of $z$ between any two copies of $y$ in $w$, then $\inter_w(x, z) \geq \lfloor \inter_w(x, y)/2\rfloor$.
	\end{lemma}
	
	\begin{proof}
		We can assume that $\inter_w(x, y)\geq 2$, as otherwise the statement holds trivially. 
		By definition, $w$ must contain a subword $w'  = xy\dots xy$ which is the concatenation of $\inter_w(x, y)$ copies of $xy$. 
		We consider the factors of $w$ between successive $y$s from $w'$; specifically, $A_1$ is the factor before the first $y$, 
		$A_2$ is the factor between the first and second $y$, and so on, up to $A_{\inter_w(x, y)}$, which is the factor between the penultimate and the final $y$. 
		We note that each of the odd numbered factors contains an $x$ (the ones from $w'$), while each of the even numbered factors contains a $z$ (by assumption). 
		It follows that $w$ contains a subword consisting of 
		$\lfloor \inter_w(x, y)/2\rfloor$ copies of $xz$ concatenated together, as required.
	\end{proof}

    \cref{lem-decoders-new} below is the culmination of our set-up, and the core technical part of the proof. The reader may find the following intuitive explanation of it useful. We start with a word $w$ realising a graph $G$. Assuming that $w$ has a suitably long factor $w'$ that uses $t$ letters (for a given $t$), we wish to find some certificate that $G$ cannot be letter-critical. By using an inductive hypothesis (the details of which are in Section~4.2), we can assume that $w'$ contains a long ``universal subword'' for $t$-letter graphs, so we will take this assumption as part of the statement of the lemma. We think of this subword as being coloured blue, of the rest of $w'$ as being coloured red, and of the rest of $w$ as being coloured black (and the vertices of $G$ inherit the corresponding colours from this particular letter realisation). The essence of the lemma then says that any decoder representing the graph induced by the blue and black vertices has ``sufficient power'' to represent the full graph $G$ (with the red vertices included) at no additional cost. Thus the presence of the superfluous red vertices is our certificate that $G$ is not letter-critical.
	
	To simplify the statement of \cref{lem-decoders-new}, let us introduce one more piece of notation. 
	For a word realisation $w=w_1w_2\ldots w_n$ of an $n$-vertex graph $G$ over a
	decoder $D=(\Sigma,A)$ and a subword $w'=w_{j_1}w_{j_2}\ldots w_{j_s}$ of $w$ (with a fixed embedding into $w$), we  denote by $\Ind_w(w')$ the set of indices $\{j_1,j_2,\ldots, j_s\}$.

	\begin{lemma}\label{lem-decoders-new}
		Let $\Sigma=\{a_1,\ldots, a_k\}$ and $\Sigma'=\{b_1,\ldots, b_m\}$ be two alphabets, and let $D$ and $D'$ be two decoders over $\Sigma$ and $\Sigma'$, respectively.  
		Suppose that a graph $G$ has a letter realisation $(\ell, c)$ over $D$ 
		such that $w:=w(\ell,c)$ has a factor $w_jw_{j+1}\ldots w_{j+p}$
            that (1) contains exactly $t$ pairwise distinct letters $a_{s_1}, a_{s_2}, \ldots, a_{s_t}$;
		and (2) contains a periodic subword $\beta$ with a period $a_{s_1} a_{s_2} \ldots a_{s_t}$ concatenated $2^{t-1}+1$ times.
		Call the vertices in $B := \{x: c(x) \in \Ind_w(\beta)\}$ blue and the vertices in $R := \{x: c(x) \in I \setminus \Ind_w(\beta)\}$ red, where $I = \{j,j+1,\ldots,j+p\}$. If the graph $G' := G \setminus R$ has a letter realisation $(\ell', c')$ over $D'$ such that blue vertices with the same letter in $(\ell, c)$ are assigned the same letter in $(\ell', c')$, then $G$ has a letter realisation $(\ell'', c'')$ over $D'$.
	\end{lemma}
	
	\begin{proof}
		Without loss of generality, assume that $\{ s_1, s_2, \ldots, s_t \}~=~[t]$ and
		denote $S := \{a_1, \dots, a_t\}$. Let $\varphi: S \to \Sigma'$ be the map such that 
		$ \varphi(\ell(x)) := \ell'(x)$, for each blue vertex $x \in B$. 
		Without loss of generality we may assume that $\varphi$ is injective by introducing twin vertices in $D'$, if necessary; by \cref{obs:twin-decoder}, any graph admitting a letter realisation over the modified $D'$ also does so over the original $D'$.  
		Thus, without loss of generality we further assume that $\varphi(a_i) = b_i$ for every $i \in [t]$.
		
		Let $C_1, \dots, C_r$ be the connected components of the asymmetry graph of $D[S]$. For every $i \in [r]$, we denote by $\Gamma_i$ the
		set of letters in $\Sigma'$ corresponding to the letters of the connected component $C_i$, \ie, $\Gamma_i := \{ \varphi(a) ~|~ a \in V(C_i) \}$.
		Let $\beta'$ denote the \emph{blue subword} of $w(\ell', c')$, \ie, the subword consisting of the letters representing the blue vertices of $G \setminus R$.
		
		\begin{claim}\label{claim:interlace}
			For any $q \in [r]$, any two distinct letters in $\Gamma_q$ interlace in $w(\ell', c')$.        
		\end{claim}
		\begin{pocd}{\ref{claim:interlace}}
			Let $b_i$ and $b_j$ be two distinct letters in $\Gamma_q$, and let $a_i$ and $a_j$ be their preimages in $V(C_q)$. 
			We will show, by induction on the distance $d$ between $a_i$ and $a_j$ in
			the underlying graph of $C_q$, that $\inter_{\beta'}(b_i,b_j) \geq 2^{t-d}$.
			Since $C_q$ has at most $|S| = t$ vertices and $\beta'$ is a subword of $w(\ell',c')$, this will imply that $b_i$ and $b_j$ interlace in $w(\ell', c')$.
			
			If $d=1$, then $a_i$ and $a_j$ are dependent in $D$, and hence the subword of $\beta$ consisting of the $(2^{t-1} + 1)$-fold concatenation of $a_ia_j$
			corresponds to a non-homogeneous graph. Since the same graph is represented by $b_i$ and $b_j$ in $w(\ell', c')$, the letters $b_i$ and $b_j$ are dependent in $D'$, and  $\beta'$ contains a subword which is the $(2^{t-1} + 1)$-fold concatenation of either $b_ib_j$ or $b_jb_i$; in particular, $\inter_{\beta'}(b_i,b_j) \geq 2^{t-1}$. 
			Furthermore, there is a $b_i$ between any two $b_j$s, and a $b_j$ between any two $b_i$s in $\beta'$. 
			
			Suppose, now that $d > 1$, and let $a_k$ be the vertex in $C_q$ that is 
			the neighbour of $a_j$ on some shortest path from $a_i$ to $a_j$ in the underlying graph of $C_q$.
			Since the distance between $a_i$ and $a_k$ is $d-1$, by the induction hypothesis,
			$\inter_{\beta'}(a_i,a_k) \geq  2^{t-(d-1)}$. Furthermore, since the distance between
			$a_j$ and $a_k$ is one, by the above base case, there is a $b_j$ between
			any two $b_k$s in $\beta'$. Hence, by Lemma~\ref{lem:interlacing},
			$\inter_{\beta'}(a_i,a_j) \geq \lfloor \inter_{\beta'}(a_i,a_k)/2 \rfloor \geq  2^{t-(d-1)}/2 = 2^{t-d}$.
		\end{pocd}
		
		\begin{claim}\label{claim:2}
			For each $q \in [r]$, the restriction of $\varphi$ on the vertices of $C_q$ either preserves all arcs, or reverses all of them. More formally, either
			\begin{enumerate}
				\item for any two distinct $a_i,a_j \in V(C_q)$ it holds that $(b_i, b_j) \in A(D')$ if and only if $(a_i,a_j) \in A(D)$; or
				
				\item for any two distinct $a_i,a_j \in V(C_q)$ it holds that $(b_i, b_j) \in A(D')$ if and only if $(a_j,a_i) \in A(D)$;
			\end{enumerate}
			
		\end{claim}
		\begin{pocd}{\ref{claim:2}}
			First, combining Claim~\ref{claim:interlace} and Observation~\ref{obs:dependency} yields that any two distinct $a_i, a_j \in V(C_q)$ are independent in $D$ if and only if
			their images $b_i$ and $b_j$ are independent in $D'$. Furthermore, since, by \cref{claim:interlace}, the images of any two distinct letters in $V(C_q)$ interlace in $w(\ell',c')$, we conclude that $\varphi$ preserves arcs between independent pairs of letters. 
			
			Now, to prove that either all arcs in $C_q$ are preserved, or all arcs in $C_q$ are reversed, it is enough to show that this holds for the arcs incident to any fixed letter $a_i \in V(C_q)$. 
			Let $a_j \in V(C_q)$ be an arbitrary neighbour of $a_i$ in the underlying graph of $C_q$. 
			Let $v_1$ and $v_2$ be the two vertices of $G \setminus R$ that are represented by the two leftmost
			copies of $a_i$ in $\beta$, and let $u$ be the vertex in $G \setminus R$ that is represented by the 
			letter $a_j$ in $\beta$ that appears between these two copies of $a_i$.
			In order to preserve the adjacencies of $u$ with $v_1$ and $v_2$ in the representation over
			$D'$, the arc between $b_i$ and $b_j$ should be reversed if and only if the order of the letters in $w(\ell',c')$ representing $v_1$ and $v_2$ is swapped with respect to the 
			order of the letters in $w(\ell,c)$ representing the same vertices.
			Since $a_j$ was chosen arbitrarily, we conclude that either all arcs incident to $a_i$ are preserved or all of them are reversed. 
			This concludes the proof of Claim~\ref{claim:2}.
		\end{pocd}
		
		Since $a_i$ and $b_i$ represent the same blue vertices in $G \setminus R$,
		and there is more than one blue vertex represented by $a_i$, we conclude that
		$(a_i,a_i) \in A(D)$ if and only if $(b_i,b_i) \in A(D')$, for every $i \in [t]$.
		This together with \cref{claim:2} imply the following
		
		\begin{corollary}\label{cor:2}
			Let $q \in [r]$, and let $w_1 w_2 \cdots w_n$ be 
			a word realisation of a graph $H$ over $C_q$. 
			Then 
			$\varphi(w_1) \varphi(w_2)  \cdots \varphi(w_n)$ or $\varphi(w_n) \varphi(w_{n-1})  \cdots \varphi(w_1)$  is a word realisation of $H$ over $D'[\Gamma_q]$.
		\end{corollary}
		
		Let $B_i \subseteq B$ be the set of blue vertices whose letters (in representation $(\ell',c')$) are in $\Gamma_i$.
		Any vertex which is neither blue nor red will be called {\it black}.
		
		\begin{claim}\label{claim:distinguish}
			Let $i \in [r]$, $x \in V(G') \setminus B_i$, and $y,z \in B_i$. If $\ell'(y)=\ell'(z)$,
			then $x$ does not distinguish $y$ and $z$, \ie, $x$ does not distinguish
			any two vertices of $G'$ from $B_i$ that are assigned the same letter.
		\end{claim}
		\begin{pocd}{\ref{claim:distinguish}}
			Let $b \in \Gamma_i$ be the letter of $y$ and $z$ assigned by $\ell'$.
			Due to injectivity of $\varphi$, we know that $y$ and $z$ are assigned by $\ell$ the same letter, say, $a \in V(C_i)$.
			
			Now, if $x \in V(G') \setminus B$, \ie, $x$ is a black vertex, then it cannot distinguish any two blue vertices assigned the same letter as the blue and the red vertices of $G$ form a factor in $w(c,\ell)$.
			If $x \in B_j$ for some $j \in [r] \setminus \{i\}$, then $x$ does not distinguish $y,z$ as
			in $w(\ell, c)$ the letter of $x$ and the letter of $y,z$ are in different connected components,
			namely in $C_i$ and in $C_j$, respectively, of the asymmetry graph of $D[S]$.
		\end{pocd}
		
		We will say that two vertices $x,y \in V(G')$ {\em commute} if $\ell'(x)$ and $\ell'(y)$ are independent in $D'$. 
		Observe that if $x$ and $y$ commute and their letters stand next to each other in a word representation of $G'$ over $D'$, then by swapping the letters we obtain a new word representation of $G'$ over $D'$. We will call such a swapping of independent neighbouring letters a \emph{transposition}.
		
		\begin{claim}\label{claim:3}     
			Let $i \in [r]$ and $x \in V(G') \setminus B_i$. Then $x$ is
			\begin{itemize}
				\item[(1)]  \emph{left-commuting with $B_i$}, \ie, $c'(x) < c'(y)$ for every $y \in B_i$ with which $x$ does not commute; or
				\item[(2)] \emph{right-commuting with $B_i$}, \ie, $c'(y) < c'(x)$ for every $y \in B_i$ with which $x$ does not commute.
			\end{itemize}
		\end{claim}
		\begin{pocd}{\ref{claim:3}}
			Assume to the contrary that there exist $y,z \in B_i$ such that $c'(y) < c'(x) < c'(z)$ and 
			$x$ commutes with none of them.
			We claim that $\ell'(y) \ne \ell'(z)$. Indeed, if $\ell'(y) = \ell'(z)$, then the non-commutativity would imply that $x$ distinguishes $y$ and $z$, which is not possible by \cref{claim:distinguish}.
			
			For the same reason, all occurrences of letter $\ell'(y)$ in $w(\ell',c')$ appear before $c'(x)$ and all occurrences of letter $\ell'(z)$ in $w(\ell',c')$ appear after $c'(x)$. However, this contradicts the fact that $\ell'(y)$ and $\ell'(z)$ interlace in $w(\ell', c')$ due to \cref{claim:interlace}.  
		\end{pocd}

		It follows from \cref{claim:3} that if a vertex is left-commuting (respectively, right-commuting) with $B_i$, then, in any word realisation of $G'$, its letter can be moved to the left (respectively, to the right) by a transposition with any letter representing a vertex in $B_i$.
		Note that a vertex can be both left- and right-commuting.

		\begin{claim}\label{claim:non-right}
			Let $i \in [r]$ and $x,y \in V(G') \setminus B_i$.
			If $c'(x) < c'(y)$ and $y$ is \emph{not} right-commuting with $B_i$, then $x$ and $y$ are left-commuting with $B_i$.
			Similarly, if $x$ is \emph{not} left-commuting with $B_i$, then $x$ and $y$ are right-commuting with $B_i$.
		\end{claim}
		\begin{pocd}{\ref{claim:non-right}} 
			The two statements are analogous, so we only show the first one. 
			If $y$ is not right-commuting with $B_i$, then, by \cref{claim:3}, it is left-commuting with $B_i$.
			To show that $x$ is also left-commuting with $B_i$, assume the contrary:
			there exists a vertex $u \in B_i$ such that $c'(u) < c'(x)$ and $x$ does not commute with $u$. 
			In particular, this implies that $c'(w) < c'(x)$ for all $w \in B_i$ with $\ell'(w) = \ell'(u)$, since $x$ does not distinguish $u$ and $w$ by \cref{claim:distinguish}.
			Similarly, since $y$ is not right-commuting with $B_i$ by assumption, there exists a vertex $v \in B_i$ with $c'(v) > c'(y)$ and with which $y$ does not commute, 
			and likewise, $c'(w) > c'(y)$ for all $w \in B_i$ with $\ell'(w) = \ell'(v)$. But this is a contradiction, because if $\ell'(u) \ne \ell'(v)$, then they must interlace by Claim~\ref{claim:interlace}, and if $\ell'(u) = \ell'(v)$, then $x$ distinguishes two vertices $u, v \in B_i$ that are assigned the same letter, which is not possible by \cref{claim:distinguish}. 
		\end{pocd}
		
		\medskip
		
		We now describe how to use transpositions to modify $(\ell', c')$ (or, more specifically, $c'$) to obtain a more well-behaved letter realisation of $G'$. 
		
		We will perform $r$ steps each consisting of a sequence of transpositions such that, after step $i$, $c'(B_1), \dots, c'(B_i)$ are \emph{intervals} in $c'$, i.e., for every $j \in [i]$, the set $c'(B_j) : = \{ c'(x) : x \in B_j \}$ is a set of consecutive integers.
		Assume that we have performed the first $i - 1$ steps so that $c'(B_1), \dots, c'(B_{i - 1})$ are intervals. Step $i$ is then done in two phases. First, we select a {\em middle}: an index $m_i$ such that everything smaller is left-commuting with $B_i$, and everything larger is right-commuting with $B_i$. A little bit of care is required, since we want to guarantee that this middle does not fall inside any previously obtained intervals $c'(B_j)$ with $j < i$. In the second phase, we simply use commutativity in order to move everything smaller than the middle to the left of $B_i$, and everything larger to the right. In detail, Step $i$ is as follows:
		
		\begin{itemize}
			\item {\em Selecting the middle:} If every vertex of $V(G') \setminus B_i$ is left-commuting with $B_i$, we define the middle $m_i$ to be $|c'(V(G'))| + 1$. Otherwise, we let $x_i$ be the leftmost vertex under $c'$ in $V(G') \setminus B_i$ which is not left-commuting with $B_i$. If $x_i \notin B_j$ for any $j \in [i - 1]$, then we set $m_i := c'(x_i)$. Otherwise, $x_i \in B_j$ for a unique $j \in [i - 1]$, and we set $m_i := \min(c'(B_j))$. We then define $L_i := \{x \in V(G') \setminus B_i : c'(x) < m_i\}$, and $M_i := \{x \in V(G') \setminus B_i : c'(x) \geq m_i\}$. 
			
			By construction, every $x \in L_i$ is left-commuting with $B_i$, since it is to the left of $x_i$. We claim that, additionally, every $x \in M_i$ is right-commuting with $B_i$. This is obviously true if $m_i = |c'(V(G))| + 1$ (since then $M_i = \emptyset$), and it is true by Claim~\ref{claim:non-right} if $m_i = c'(x_i)$. It remains to show that this is the case when $x_i \in B_j$ for some $j \in [i - 1]$. It suffices to show that $x$ is right-commuting with $B_i$ for all $x$ with $m_i \leq c'(x) < c'(x_i)$, since again by Claim~\ref{claim:non-right}, every $x$ with $c'(x) \geq c'(x_i)$ is right-commuting with $B_i$. 
			
			To show this, we first note that whether or not a vertex in $B_j$ left-commutes (respectively, right-commutes) with $B_i$ only depends on its letter. Indeed, since $c'(B_j)$ is an interval,
			there are no elements of $B_i$ between any two $x, y \in B_j$, and hence if $\ell'(x) = \ell'(y)$, then $x$ left-commutes (respectively, right-commutes) with $B_i$ if and only if $y$ does.
			In particular, since $x_i$ is not left-commuting with $B_i$, and any $x \in V(G') \setminus B_i$ with $c'(x) < c'(x_i)$ is left-commuting with $B_i$, this implies that $\ell'(x) \neq \ell'(x_i)$.
			Thus, by Claim~\ref{claim:interlace}, for any $x \in B_j$ with $c'(x) < c'(x_i)$, the letters $\ell'(x)$ and $\ell'(x_i)$ interlace. Hence, there is a $y \in B_j$ with $\ell'(y) = \ell'(x)$ and $c'(y) > c'(x_i)$. Consequently, Claim~\ref{claim:non-right} implies that $y$ right-commutes with $B_i$, and so $x$ does as well, by the argument above. This establishes that each $x \in L_i$ left-commutes with $B_i$, and each $x \in M_i$ right-commutes with $B_i$, as desired

			\item {\em Separating the word:} With $L_i$ and $M_i$ as defined above, using transpositions, we modify $c'$ by moving to the left, one by one and starting with the leftmost one, each vertex in $L_i$ until it has no more vertices in $B_i$ to its left. Proceed similarly by moving to the right the vertices in $M_i$, starting with the rightmost one. We note that this procedure does not modify the graph represented by $(\ell', c')$, since it consists of a succession of transpositions. Furthermore, notice that in Step $i$, the relative order under $c'$ of the vertices in $B_i$ does not change, and similarly the relative order   of the vertices not in $B_i$ also does not change. By the choice of $m_i$, each $c'(B_j)$ with $j \in [i - 1]$ is still an interval, and moreover, we observe that $c'(B_i)$ is now also an interval. 
		\end{itemize} 
		
		\noindent
		After Step $r$, our procedure yields that:
		
		\begin{itemize}
			\item [(*)] For each $i \in [r]$, the set $c'(B_i)$ is an interval.
		\end{itemize}  
		
		\medskip 
		
		\noindent
		Now, we are finally ready to construct the desired letter realisation $(\ell'', c'')$ of $G$ starting with the modified letter realisation $(\ell', c')$ of $G'$ that satisfies (*).
		Let $R_i \subseteq R \subseteq V(G)$ be the set of red vertices whose letters are in $C_i$, and
		let $w_1 w_2 \ldots w_{s_i}$ be the subword of $w(\ell, c)$ corresponding to 
		the vertices in $B_i\cup R_i$. In particular, $w_1 w_2 \ldots w_{s_i}$ is a word realisation of $G[B_i\cup R_i]$ over $D[C_i]$.
		To obtain $(\ell'', c'')$, we simply replace the factor corresponding to $B_i$ in $w(\ell',c')$ by the word realisation 
		of $G[B_i \cup R_i]$ over $D'[\Gamma_i]$ as in Corollary~\ref{cor:2}, \ie, either by $\varphi(w_1) \varphi(w_2) \ldots \varphi(w_{s_i})$ or its reverse. We do this for each $i$. Let us denote by $G''$ the graph realised by $(\ell'', c'')$ over $D'$. 
		We claim that $G=G''$. Clearly, both graphs have the same vertex set. Thus, we only need to check that adjacencies between the vertices are the same in both graphs.
		First, we note that since $c'(B_i)$, $i \in [r]$, are intervals in $c'$, and $c''(B_i \cup R_i)$, $i \in [r]$, are intervals in $c''$, and both $\ell'$ and $\ell''$ assign the same letters to the same blue vertices, the addition of the vertex sets $R_i$, $i \in [r]$, does not change the adjacencies in $G' = G'' \setminus R = G \setminus R$.
		It remains to examine the adjacencies between pairs of vertices in which at least one vertex is in $R$.

		\begin{itemize}
			\item Let $x \in R_i$ and $y \in B_i \cup R_i$ for some $i \in [r]$. Then the adjacency between $x$ and $y$ is the same in both graphs by construction. 
			
			\item Let $x \in R_i$ and $y \in B_j \cup R_j$ for distinct $i,j \in [r]$. Denote by  $b \in \Gamma_i$ the letter representing $x$ and by $b' \in \Gamma_j$
			the letter representing $y$ in $G''$. The letters  $\varphi^{-1}(b)$ and $\varphi^{-1}(b')$ are independent in $D$, and hence the adjacency in $G$ of any two vertices (in particular, of $x$ and $y$) labelled by these letters does not depend on their relative positions in  $w(\ell,c)$. 
			We observe that in letter realisation $(\ell'', c'')$,  $B_i$ necessarily contains a vertex $x'$ labelled by $b$ and $B_j$ necessarily contains a vertex $y'$ labelled by $b'$ (where $y'$ could coincide with $y$). Furthermore, with respect to $c''$, either all vertices in $R_i \cup B_i$ are to the left of all vertices in $B_j \cup R_j$, or all vertices in $R_i \cup B_i$ are to the right of all vertices in $B_j \cup R_j$.
			This implies that, in $G''$, vertices $x$ and $y$ have the same adjacency as $x'$ and $y'$. In turn, since $x'$ and $y'$ are vertices of $G'$, their adjacency in $G''$ is the same as in $G$, and coincides with the adjacency of $x$ and $y$ in $G$.
			Consequently, $x$ and $y$ have the same adjacency in both $G''$ and $G$.
			
			\item Let $x \in R_i$ for some $i \in [r]$, and let $y$ be a black vertex, \ie, $y \in V(G'') \setminus (B \cup R)$. Note that in each $c$ and $c''$, vertex $y$ is either to the left or to the right of all vertices in $B_i \cup R_i$. Thus in both graphs $G$ and $G''$ vertex $y$ does not distinguish $x$ from any other vertex in $B_i \cup R_i$ that is labelled by the same letter as $x$.			
			
			Since $B_i$ necessarily contains a vertex $x'$ labelled by $\ell''(x)$, and $y$ and $x'$ are  vertices of $G'$, the adjacency between $y$ and $x'$ in $G''$ is the same as in $G$. This implies that the adjacency between $y$ and $x$ is the same in both graphs too.
		\end{itemize}
		We have now completed the proof by showing that $(\ell'',c'')$ is a letter realisation of $G$. 
	\end{proof}

	\subsection{An upper bound on the size of obstructions}
	\label{sec:upper-bound}
	
	Let $G$ be a critical $k$-letter graph with a $k$-letter realisation $(\ell, c)$.
	We define $g(t,\ell, c)$ as the maximum length of a factor of $w(\ell,c)$
	in which at most $t$ distinct letters appear. Our first goal is to show that the function $g$ is bounded above by a function which only depends on $t$ and $k$. 
	In other words, we will show that the value
	$$
	f_k(t) := \max\{g(t,\ell, c): (\ell, c)\text{ is a $k$-letter realisation} \text{ of a critical $k$-letter graph}\}
	$$ 
	is finite for each $t$ and $k$. Assuming this, $f_k(t)$ is then a function such that, for any critical $k$-letter graph $G$, and any $k$-letter representation $(\ell, c)$ of $G$ over
	any decoder $D$, any factor of length greater than $f_k(t)$ must contain at least $t + 1$ different letters. In particular, $f_k(k)$ is an upper bound on
	the number of vertices in any critical $k$-letter graph. 
	
	\begin{lemma}\label{prop:no-large-factor}
		For every natural $k$, the following holds:
		\begin{enumerate}
			\item[(1)] $f_k(1)\leq 3$;
			\item[(2)] $f_k(t) < (2^{t-1}+1) \cdot (k-1)^t \cdot t  \cdot (f_k(t-1) + 1)$, for every $2\leq t \leq k$;
			\item[(3)] $f_k(k) = 2^{O(k^2\log k)}$.
		\end{enumerate}
	\end{lemma}
	\begin{proof}
		We start by proving (1). Observe that a factor of a word realisation of a graph in which every letter is the same corresponds to a twin class in the graph.
		This together with \cref{lem:large-twin-classes} imply that no word realisation of a $k$-critical graph contains a factor of lengths 4 consisting of the same letter.
		
		We now proceed with the proof of (2). Assume it does not hold and 
		let $(t,k)$ be a lexicographically minimal pair with $2 \leq t \leq k$ for which the proposition does not hold, \ie, there exists a critical $k$-letter graph $G$
		and a $k$-letter realisation $(\ell, c)$ of $G$ over a $k$-letter decoder $D$ such that $w := w(\ell,c)$ has a factor $u$ of length $(2^{t-1}+1) \cdot (k-1)^t \cdot t \cdot (f_k(t-1) + 1)$ that contains exactly $t$ different letters, say, $a_1, \ldots, a_{t}$.
		
		Let $\gamma := (2^{t-1}+1) \cdot (k-1)^t$ and let $u^1, u^2, \ldots, u^\gamma$
		be consecutive factors of $u$ each of length $t \cdot (f_k(t-1)+1)$.
		We note that, by the minimality of the counterexample, each of these factors contains a subword $a_1\cdots a_{t}$, since each of its subfactors of length $f_k(t-1) + 1$ must contain every one of the $t$ letters. 
		In each factor $u^j$, pick one such subword $a_1\cdots a_{t}$ and denote it by $s^j$.
		
		Consider the graph $G^+ := G \setminus \left\{x: c(x) \in I \setminus \left(\bigcup_j \Ind_w(s^j)\right)\right\}$, where $I = \Ind_w(u)$.
		By assumption, $G^+$ has lettericity at most $k - 1$, and therefore, it has a letter
		realisation $(\ell^{+}, c^{+})$ over a $(k-1)$-letter decoder $D'$. In particular, a simple application of the pigeonhole principle allows us to conclude that there are $2^{t-1}+1$ of the
		words $s^j$ in which $a_i$ is mapped in $(\ell^{+}, c^{+})$ to the same letter $b_i$ of $D'$ for all $i \in [t]$. Let $\beta$ be the subword of $w$ consisting of all these $2^{t-1}+1$ words $s^j$.
		
		Following the notation from the statement of Lemma~\ref{lem-decoders-new}, write 
		$B = \{x \in V(G): c(x) \in \Ind_w(\beta)\}$,
		$R = \{x \in V(G): c(x) \in I \setminus \Ind_w(\beta)\}$, and
		$G' := G \setminus R$. Letting $(\ell', c')$ be the restriction of $(\ell^{+}, c^{+})$ to $G'$, we see that the conditions of Lemma~\ref{lem-decoders-new}, are satisfied. Hence, by \cref{lem-decoders-new}, there exists a letter realisation of $G$ over the $(k-1)$-letter decoder $D'$, which is in contradiction with the fact that $G$ has lettericity $k$. This concludes the proof of (2).
		
		Finally, we are ready to prove (3). For every $2 \leq t \leq k$, we derive from (2) the following inequality
		$$
		f_k(t) < (2^{t-1}+1) \cdot (k-1)^t \cdot t \cdot (f_k(t-1) + 1) \leq (2k)^{k+1} \cdot f_k(t-1).
		$$ 
		Now, from this inequality and (1), we deduce the desired 
		$$
		f_k(k) \leq 3 \cdot ((2k)^{k+1})^{k - 1} \leq  (2k)^{k^2} = 2^{O(k^2\log k)}.
		$$
	\end{proof}

	We finish by deriving the intended upper bound on the size of obstructions for $k$-lettericity.
	
	\begin{theorem} \label{th:main}
		Let $k$ be a positive integer. 
		Any obstruction for $k$-lettericity has at most $2^{O(k^2\log k)}$ vertices.
	\end{theorem}
	
	\begin{proof} Let $G$ be an obstruction for $k$-lettericity.
		From Lemma~\ref{lem:chain-graph}, it follows that $\lettericity(G) \leq 2k+1$.
		Hence, $G$ is a critical $s$-letter graph for some $s \leq 2k+1$.
		Thus, the number of vertices in $G$ is 
		upper bounded by $f_{s}(s)$, which, by \cref{prop:no-large-factor} (3), is bounded
		by $2^{O(s^2\log s)} = 2^{O(k^2\log k)}$.
	\end{proof}

	\section{Conclusion}
    \label{sec:con}
	We conclude our paper with a number of open problems. 

    \medskip
    \noindent
    {\it Algorithmic problems}. In spite of the positive results presented in the paper, the complexity of computing the lettericity, or equivalently, the complexity of recognizing $k$-letter graphs when $k$ is not fixed, remains unknown. We conjecture it to be NP-complete. 

	\medskip
    \noindent
    {\it Bounds on obstructions}.
	We note that there is still a big discrepancy between our bound and the obstruction sizes we have observed so far, which are all linear. That is to say, all of our ``standard'' examples of structures with growing lettericity have size within a small constant factor of the lettericity. This is really curious, since even our exponential bound required significant effort to derive. In \cite{AlecuL2021}, the structure of obstacles to bounded lettericity is investigated, and certain conjectures are made. In view of those conjectures, it seems believable (albeit perhaps slightly ambitious) to conjecture that this is, in fact, the true answer: 
	
	\begin{conjecture}
		There is a $C > 0$ such that any obstruction for the class of $k$-letter graphs has at most $C \cdot k$ vertices.
	\end{conjecture}

    \medskip
    \noindent
    {\it Graphs of unbounded lettericity}.
    Finally, we observe that lettericity can be of interest even for classes, 
    where this parameter is unbounded. For instance, the class of bipartite permutation
    graphs is precisely the class of letter graphs deciphered by the decoder $D$ which is 
    an infinite directed path: these graphs are also known as Parikh word
    representable graphs \cite{Parikh}. By extending $D$ to its transitive closure, we extend 
    the class of bipartite permutation graphs to the class of all permutation graphs. Both these 
    classes play a critical role in the area of graph parameters. In particular, the class of bipartite
    permutations graphs is a minimal hereditary class of unbounded clique-width, while 
    the permutation graphs form a minimal hereditary class of unbounded twin-width. Identifying 
    other critical classes by means of letter graphs with a well-structured infinite decoder 
    is an interesting research direction. 

    \medskip
    \medskip
    \noindent
    \emph{\bf Acknowledgement.} This work is a follow-up of \cite{AlecuFKLVZ2022} and the authors thank Robert Ferguson and Vince Vatter for preliminary discussions.

    \medskip
    \medskip
    \noindent
    \emph{\bf Conflict of interest statement.}
    The authors declare no conflicts of interest.

	\newpage
	
	\appendix

	\section{An FPT algorithm with a moderate exponential on the parameter}\label{sec:fpt-recognition}

	In this section, we will develop a Dynamic Programming algorithm which decides whether a given graph has lettericity at most $k$ in time
	$2^{O(k^2\cdot 2^{2k})}\cdot n^3$. We first recall the following algorithm that decides whether a graph has linear rank-width at most $k$, and if so, outputs a
	linear ordering of optimal width. 
	
	\begin{theorem}[\cite{JeongKO17}]\label{thm:recognition-lrwd} Let $k$ be a positive integer. There is an algorithm that takes as input an $n$-vertex graph $G$
		and checks in time $2^{O(k^2)}\cdot n^3$ whether $G$ has linear rank-width at most $k$ and if so, outputs a linear ordering of $G$ of width at most $k$.
	\end{theorem}

	We start with some preliminaries. Let $G$ be a graph. If $B\subseteq V(G)$, then we define the binary relation $\equiv_G^B$ on subsets of $B$, where for any
	$X,Y\subseteq B$, we have
	\begin{align*}
		X\equiv_G^B Y\ \textrm{if for any $z\in V(G)\setminus B$, we have}\  \min\{|N_G(z)\cap X|,1\} = \min\{|N_G(z)\cap Y|,1\}.
	\end{align*}
	
	In other words, $X\equiv_G^B Y$ if any vertex $z\in V(G)\setminus B$ that has a neighbour in $X$, also has a neighbour in $Y$, and vice versa. It is not hard to
	check that $\equiv_G^B$ is an equivalence relation. Whenever $\{x\} \equiv_G^B \{y\}$, we call $x$ and $y$ \emph{twins \wrt $B$}, and call \emph{twin class of
		$B$} a maximal set of vertices that are all pairwise twins \wrt $B$. We denote by $\Twin(B)$ the set of twin classes of $B$. 
	
	For a subset $B$ of $V(G)$ and subsets $X$ and $Y$ of $B$, we call $X$ and $Y$ \emph{equivalent}, denoted by $\equiv_{G,\compl{G}}^A$, if $X\equiv_G^A Y$ and
	$X\equiv_{\compl{G}}^A Y$. It is not hard to check that $\equiv_{G,\compl{G}}^B$ is an equivalence relation. Moreover, the checking of the following can be
	found in \cite{BuiXuanTV2013}.
	
	\begin{lemma}[\cite{BuiXuanTV2013}]\label{lem:equiv} Let $B$ and $B'$ be vertex subsets of $V(G)$ with $B\subseteq B'$, for some graph $G$. For any two subsets $X$ and $Y$ of
		$B$, if $X\equiv_G^B Y$, then $X\equiv_G^{B'} Y$.
	\end{lemma}
	
	Let $D=(\alphabet,A)$ be a decoder. If $\bar{B}=(B_1,\ldots,B_k)$ is an ordered partition of $B\subseteq V(G)$, we define $\paths(\bar{B})$ as the
	set of tuples $(i,j,T,T')\in [k]^2\times (\Twin({B}))^2$ such that there are vertices $x\in B_i$ and $y\in B_j$ and a directed path from $x$ to $y$ in
	$\CG(G[B],D,\ell_{\bar{B}})$, and $T$ and $T'$ are, respectively, the twin classes \wrt $B$ of $x$ and $y$.
	
	The following lemma tells that we can restrict the set of partial information to keep at each step of the DP algorithm.
	
	\begin{lemma}\label{lem:correct-dp} Let $D=(\alphabet,k)$ be a decoder. Let $B$ a subset of the vertex set of a graph $G$, and let $x\in V(G)\setminus B$. Let $\bar{X}=(X_1,\ldots,X_k)$ and
		$\bar{Y}=(Y_1,\ldots,Y_k)$ be two ordered partitions of $B$ such that:
		\begin{enumerate}
			\item[(i)] $X_i\equiv_{G,\compl{G}}^B Y_i$ for all $1\leq i\leq k$,
			\item[(ii)] $\paths(\bar{X}) = \paths(\bar{Y})$.
			\item[(iii)] $\bar{X}$ and $\bar{Y}$ are both letter partitions of $G[B]$ over $D$.  
		\end{enumerate}
		Then, for each $1\leq i \leq k$, we have
		\begin{enumerate}
			\item $X_j\equiv_{G,\compl{G}}^{B'} Y_j$ and $X_j\cup \{x\} \equiv_{G,\compl{G}}^{B'} Y_j\cup \{x\}$,  for all $1\leq j\leq k$,
			\item $\paths(\bar{X}^i) = \paths(\bar{Y}^i)$.
			\item $\bar{X}^i$ is a letter partition of $G[B]$ over $D$ if and only if $\bar{Y}^i$ is.
		\end{enumerate}
		where $B'=B\cup \{x\}$, $\bar{X}^i=(X_1,\ldots, X_{i-1},X_i\cup \{x\}, X_{i+1},\ldots, X_k)$ and
		$\bar{Y}^i=(Y_1,\ldots, Y_{i-1},Y_i\cup \{x\}, Y_{i+1},\ldots, Y_k)$.
	\end{lemma}

	\begin{proof} Using Lemma \ref{lem:equiv}, we can conclude that $X_j\equiv_{G,G'}^{B'} Y_j$ and $X_j\cup \{x\} \equiv_{G,\compl{G}}^{B'} Y_j\cup \{x\}$, for all
		$1\leq j\leq k$. 
		
		Let's now prove (ii). Because $\paths(\bar{X}) = \paths(\bar{Y})$, we need only to look at paths going through $x$. Let
		$P=(z_1,\ldots, z,x,z',\ldots,z_2)$ be a path in $\CG(D[B'],D,\ell_{\bar{X}^i})$ with $z_1\in X_t$, $z_2\in X_l$, $T_1, T,T'$ and $T_2$ the twin-classes
		of, respectively, $z_1,z,z'$ and $z_2$ \wrt $B$, \ie, $(t,\ell_{\bar{X}}(z),T_1,T), (\ell_{\bar{X}}(z'),l,T',T_2) \in \paths(\bar{X})$. Because
		$\paths(\bar{X}) = \paths(\bar{Y})$, $(t,\ell_{\bar{X}}(z),T_1,T), (\ell_{\bar{X}}(z'),l,T',T_2) \in \paths(\bar{Y})$, \ie, there are
		vertices $z_1',y,y',z_2'$ with
		\begin{itemize}
			\item $\ell_{\bar{Y}}(z_1')=t$, $\ell_{\bar{Y}}(z_2')=l$, $\ell_{\bar{Y}}(y)=\ell_{\bar{X}}(z)$, $\ell_{\bar{Y}}(y')= \ell_{\bar{X}}(z')$,
			\item paths $P_1=(z_1',\ldots, y)$ and $P_2=(y',\ldots, z_2')$ in $\CG(G[B],D,\ell_{\bar{Y}})$, and
			\item $y$ and $z$ are twins \wrt $B$, $y'$ and $z'$ are twins \wrt $B$.
		\end{itemize}
		
		Therefore, we can conclude that $(z,x)\in
		\CG(D[B'],D,\ell_{\bar{X}^i})$ (resp. $(x,z')\in \CG(D[B'],D,\ell_{\bar{X}^i})$) if and only if $(y,x)\in \CG(G[B'],D,\ell_{\bar{Y}^i})$
		(resp. $(x,y')\in \CG(G[B'],D,\ell_{\bar{Y}^i})$). We have then proved that if $(t,l,T_1,T_2)\in \CG(D[B'],D,\ell_{\bar{X}^i})$, then
		$(t,l,T_1,T_2)\in \CG(G[B'],D,\ell_{\bar{Y}^i})$. Since the other direction is symmetric, we can conclude (ii). 
		
		Let's now prove (iii). Assume that $\bar{X}^i$ is a letter partition of $G[B']$ over $D$. Since $X_j\equiv_{G,\compl{G}}^B Y_j$, for each $1\leq j\leq k$ and $\bar{Y}$ is a
		letter partition of $G[B]$ over $D$, we can conclude that and 
		
		\begin{itemize}
			\item $X_i\cup \{x\}$ is an independent set (resp. clique) if and only if $Y_i\cup \{x\}$ is, and 
			\item $G[X_i\cup \{x\},X_j]$ is a complete bipartite graph (resp. edgeless) if and only if  $G[Y_i\cup \{x\},Y_j]$ is, for each $j\in
			[k]\setminus \{i\}$.
		\end{itemize}
		
		It remains to check that $\CG(G[B'],D,\ell_{\bar{Y}^i})$ does not contain a circuit. Since $\bar{Y}$ is a letter partition  of $G[B]$ over $D$, any
		circuit in $\CG(G[B'],D,\ell_{\bar{Y}^i})$ contains $x$. So, let $C$ be a circuit in $\CG(G[B'],D,\ell_{\bar{Y}^i})$ and let $z$ and $z'$ be,
		respectively, the in- and out-neighbour of $x$ in $C$. Thus, $(\ell_{\bar{Y}}(z'), \ell_{\bar{Y}}(z), T',T)\in \paths(\bar{Y})$ with $T$ and $T'$
		the twin-classes of, respectively, $z$ and $z'$ with respect to $B$, and so $(\ell_{\bar{Y}}(z'), \ell_{\bar{Y}}(z), T',T)\in \paths(\bar{X})$. Because $x$
		is adjacent (or non-adjacent) to every vertex in $T$, and similarly for $T'$, we can conclude that there is also a circuit in
		$\CG(G[B'],D,\ell_{\bar{X}^i})$, a contradiction.
		
		For the same reasons, $\bar{X}^i$ is a letter partition of $G[B']$ over $D$ if $\bar{Y}^i$, which concludes the proof of (iii).
	\end{proof}

	We will now recall some lemmas that help in upper-bounding the time complexity of the dynamic programming algorithm. The following lemma implies that
	$\lrwd(\compl{G}) \leq \lrwd(G)+1$.
	
	\begin{lemma}\label{lem:folklore} For every graph $G$ and every $X\subseteq V(G)$, $\cutrk_{\compl{G}}(X)\leq \cutrk_G(X)+1$.
	\end{lemma}
	
	\begin{proof} It follows from the fact that $M_{\compl{G}}[X,\compl{X}]$ is obtained from $M_G[X,\compl{X}]$ by adding the column $(1,1,\cdots,1)^t$ to every
		column of $M_{\compl{G}}[X,\compl{X}]$, and this operation can only increase the rank by at most $1$.
	\end{proof}
	
	If $V(G)$ is arbitrarily linearly ordered, we denote by $\rep_G^A(X)$ the lexicographically smallest set $R\subseteq A$
	among all $R\equiv_G^A X$ of minimum size, and denote by $\Rep_G(A)$ the set $\{\rep_G^A(X)\mid X\subseteq A\}$.  In our algorithm, we will need to compute the set of
	representatives for each equivalence class, and for that we will use the following.
	
	\begin{lemma}[\cite{BuiXuanTV2013}]\label{lem:bound-necd} Let $G$ be an $n$-vertex graph. Then, for every $A\subseteq V(G)$,
		\begin{enumerate}
			\item the number of equivalence classes of $\equiv_G^A$ is bounded by $2^{\cutrk_G(A)^2}$,
			\item one can compute in time $O(|\Rep_G(A)| \cdot n^2 \cdot \log(|\Rep_G(A)|))$, the set $\Rep_G(A)$ and a data structure that, given
			a subset $X\subseteq A$, computes $\rep_G^A(X)$ in time $O(|A|\cdot n\cdot \log(|\Rep_G(A)|))$.
		\end{enumerate}
	\end{lemma}
	
	We are now ready to give the algorithm in the following theorem.

	\begin{theorem}\label{thm:recognition1}  Let $k$ be a positive integer.  There is an algorithm that takes as input an $n$-vertex graph $G$ and checks in time
		$2^{O(k^3)}\cdot 2^{O(k^2\cdot 2^{2k})}\cdot n^3$ whether $G$ has lettericity at most $k$, and if so outputs a $k$-letter realisation of $G$.
	\end{theorem}
	
	\begin{proof} 
		Let $G$ be an $n$-vertex graph. We first check whether $G$ has linear rank-width at most $k$, and if NO, we can correctly answer that the lettericity of $G$
		is strictly greater than $k$ by Proposition \ref{prop:lettericity-lrw}. We can therefore assume that $G$ has linear rank-width at most $k$ and by Theorem
		\ref{thm:recognition-lrwd} we can compute a linear ordering $x_1,x_2,\ldots,x_n$ of $G$ of width at most $k$ in time $2^{O(k^2)}\cdot n^3$.
		
		Let $D=([k],A)$ be a decoder. By Theorem \ref{thm:characterisation}, $G$ has a $k$-letter realisation over $D$ if and only if $G$ has a letter partition
		$\bar{V}=(V_1,\ldots,V_k)$ over $D$ and for any topological ordering $\pi$ of $\CG(G,D,\ell_{\bar{V}})$, the pair $(\ell_{\bar{V}},\pi)$ is a
		letter realisation of $G$. The dynamic programming algorithm will check the existence of such a letter partition and if yes, computes one.
		
		For each $1\leq i\leq n$, let $A_i=\{x_1,\ldots,x_i\}$ and let $\Rep_{G,\compl{G}}(A_i)$ be the number of equivalence classes of
		$\equiv_{G,\compl{G}}^{A_i}$. For each decoder $D=([k],A)$, $1\leq i \leq n$, and
		$(R_1,\ldots,R_k,\cC)\in (\Rep_{G,\compl{G}}(A_i))^k\times ([k]^2\times (\Twin(A_i))^2)$, let $\mathsf{Sol}_i^D[(R_1,\ldots,R_k,\cC)]$ be the
		set
		\begin{small}
			$$ \{ (B_1,\ldots,B_k)\ \textrm{letter partition of $G[A_i]$ over $D$}\mid B_j\equiv_{G,\compl{G}}^{A_i} R_j,\ \textrm{for all $1\leq j\leq k$ and $\cC = \paths(B_1,\ldots,B_k)$}\},$$
		\end{small}
		
		and let $tab_i^D[(R_1,\ldots,R_k,\cC)]$ be an arbitrary element in $\mathsf{Sol}_i^D[(R_1,\ldots,R_k,\cC)]$.  From the characterisation of
		letter realisations in Theorem \ref{thm:characterisation}, we have that $G$ is a $k$-letter graph if and only if $tab_n^D$ is non-empty for some decoder $D$.
		
		Let us now explain the computation of the tables.  The computation of $tab_1^D$ is trivial, for any decoder $D$, and suppose that we have already computed
		$tab_i$, for some $1\leq i\leq n-1$. The following algorithm computes $tab_{i+1}$
		\begin{enumerate}
			\item Pick a decoder $D$ on $k$-vertices and $(R_1,\ldots,R_k,\cC')\in (\Rep_{G,\compl{G}}(A_i))^k\times ([k]^2\times (\Twin(A_i))^2)$ such that
			$tab_i^D[(R_1,\ldots,R_k,\cC')]=(B_1,\ldots,B_k)$ (\ie, $tab_i^D[(R_1,\ldots,R_k,\cC')]\ne \emptyset$)
			\item For each $1\leq j\leq k$, if $(B_1,\ldots, B_{j-1},B_j\cup \{x_{i+1}\},B_{j+1},\ldots, B_k)$ is a letter partition of $G[A_{i+1}]$ over $D$, then 
			set $tab_{i+1}^D[(S_1,\ldots,S_k,\cC)]$ to  $(B_1,\ldots, B_{j-1},B_j\cup \{x_{i+1}\},B_{j+1},\ldots, B_k)$ where
			\begin{itemize}
				\item $R_t\equiv_{G,\compl{G}}^{A_{i+1}} S_t$ for every $t\in [k]\setminus \{j\}$, and $S_j \equiv_{G,\compl{G}}^{A_{i+1}} R_j\cup \{x_{i+1}\}$,
				\item $\cC=\paths(B_1,\ldots,B_{j-1},B_j\cup \{x_{i+1}\},B_{j+1},\ldots,B_k)$.
			\end{itemize}
		\end{enumerate}
		\noindent
		For the correctness of the computation of $tab_{i+1}^D$, first notice that each tuple set to $tab_{i+1}^D[(S_1,\ldots,S_k,\cC)]$ belongs to
		$\mathsf{Sol}_{i+1}^D[(S_1,\ldots,S_k,\cC)]$. Second, whenever there is a letter partition $(Z_1,\ldots, Z_{j-1},Z_j\cup \cup \{x\}, Z_{j+1},\ldots, Z_k)$
		of $G[A_{i+1}]$ over $D$, which belongs to $\mathsf{Sol}_{i+1}^D[(S_1,\ldots, S_k,\cC)]$, there is $(R_1,\ldots, R_k,\cC')$ such that
		$(Z_1,\ldots, Z_{j-1},Z_j,Z_{j+1},\ldots, Z_k)$ belongs to $\mathsf{Sol}_i^D[(R_1,\ldots, R_k,\cC')]$ and
		\begin{enumerate}
			\item $R_t\equiv_{G,\compl{G}}^{A_{i+1}} S_t$ for every $t\in [k]\setminus \{j\}$, and $S_j \equiv_{G,\compl{G}}^{A_{i+1}} R_j\cup \{x_{i+1}\}$,
			\item $\cC'=\paths(Z_1,\ldots,Z_{j-1},Z_j,Z_{j+1},\ldots,Z_k)$.
		\end{enumerate}
		\noindent
		We can therefore conclude using Lemma \ref{lem:correct-dp} that, after executing the algorithm above, 
		$$tab_{i+1}^D[(S_1,\ldots,S_k,\cC)]\in \mathsf{Sol}_{i+1}^D[(S_1,\ldots,S_k,\cC)]$$ for each $(S_1,\ldots,S_k,\cC)\in
		(\Rep_{G,\compl{G}}(A_{i+1}))^k\times ([k]^2\times (\Twin(A_{i+1}))^2)$. 
		
		\medskip
		Let us now analyse time complexity. Notice first that, $X\equiv_{G,\compl{G}}^{A_i} Y$ whenever
		$(\rep_G^{A_i}(X),\rep_{\compl{G}}^{A_i}(X))=(\rep_G^{A_i}(Y),\rep_{\compl{G}}^{A_i}(Y))$. So, we can assume that the equivalence class of each $X$ \wrt
		$\equiv_{G,\compl{G}}$ is exactly $(\rep_G^{A_i}(X),\rep_{\compl{G}}^{A_i}(X))$.  By Lemmas \ref{lem:bound-necd} and \ref{lem:folklore}, this number of
		equivalence classes is bounded by $2^{O(k^2)}$. So, the number of different entries of each $tab_i^D$ is bounded by
		$2^{O(k^3)}\cdot 2^{O(k^2\cdot 2^{2k})}$.

		By using Lemma \ref{lem:bound-necd} we can first compute $\Rep_{G,\compl{G}}(A_i)$, for all $1\leq i \leq n$, in time $2^{O(k^2)}\cdot n^3$, and be able to
		compute $(\rep_G^{A_i}(X),\rep_{\compl{G}}^{A_i}(X))$ in time $O(n^2\cdot k^2)$, for each $X\subseteq A_i$. The time complexity to compute $tab_{i+1}$ needs
		then $2^{O(k^3)}\cdot 2^{O(k^2\cdot 2^{2k})}\cdot n^2$. We can therefore, compute $tab_n$ in time $2^{O(k^3)}\cdot 2^{O(k^2\cdot 2^{2k})}\cdot n^3$. Notice
		that there is only one equivalence class in $\equiv_{G,\compl{G}}^{A_n}$, which is $(\emptyset,\emptyset)$, and one twin class in $\Twin(A_n)$, represented
		also by $\emptyset$. It takes then time $2^{O(k^2)}$ to decide whether there is some decoder $D$ such that
		$tab_n^D[(\emptyset,\emptyset),\ldots,(\emptyset,\emptyset),\cC]$ is non-empty, and if so, we compute the compatibility graph associated with its entry
		in time $O(k^2\cdot n^2)$ and perform a topological ordering in time $O(n^2)$.
	\end{proof}

\end{document}